\documentclass[reqno, 11pt]{amsart}

\usepackage{
	amssymb,
	amsmath,
	amsthm,
	eucal,
	dsfont,
	multicol,
	multirow,
	mathrsfs,
	tikz,
	graphicx,
	cite,
	makecell
}
\allowdisplaybreaks
\usepackage{hyperref}
\hypersetup{
	colorlinks=true,        
	linkcolor=black,        
	citecolor=blue,        
	urlcolor=black         
}
\usepackage{lipsum}
\makeatletter\renewcommand*{\eqref}[1]{%
	\hyperref[{#1}]{\textup{\tagform@{\!\!\ref*{#1}}}}%
}\makeatother 
\makeatletter

\@addtoreset{equation}{section}
\makeatother
\theoremstyle{plain}
\newtheorem{theorem}{Theorem}[section]
\newtheorem{lemma}[theorem]{Lemma}
\newtheorem{proposition}[theorem]{Proposition}

\newtheorem{corollary}[theorem]{Corollary}
\theoremstyle{definition}
\newtheorem{definition}[theorem]{Definition}

\def\R{{\mathbb{R}}}

\def\i{{\rm i}}
\def\d{{\rm d}}
\def\<{{\langle}}
\def\>{{\rangle}}

\def\i{{\rm i}}
\def\d{{\rm d}}
\def\<{{\langle}}
\def\>{{\rangle}}

\title
[Uncertainty Principle and Schr\"{o}dinger Evolutions]{Dynamical versions of Morgan's Uncertainty Principle and Electromagnetic Schr\"{o}dinger Evolutions}
\author{Shanlin Huang \quad Zhenqiang Wang}

\address{Shanlin Huang,  School of Mathematics (Zhuhai), Sun Yat-sen University, Zhuhai 519082, Guangdong, China}
\email{shanlin\_huang@hust.edu.cn}
\address{Zhenqiang Wang, School of Mathematics and Statistics, Huazhong University of Science and Technology, Wuhan 430074, Hubei, China}
\email{zhq\_wang@hust.edu.cn}

\subjclass[2020]{35B60, 35J10, 35Q41}
\keywords{Electromagnetic potentials, unique continuation, uncertainty principle}
\thanks{}

\begin{document}
	
	\begin{abstract}
		This paper investigates the unique continuation properties of solutions of the
		electromagnetic Schr\"{o}dinger equation
		$$
		\i\partial_{t}u(x,t)+(\nabla-\i A)^{2}u(x,t)=V(x,t)u(x,t)\,\,\,\, \mbox{in} \,\,\,\mathbb{R}^{n}\times [0,1],
		$$
		where $A$ represents a time-independent magnetic vector potential and $V$ is a bounded, complex valued time-dependent potential. 
		Given $1<p<2$ and $1/p+1/q=1$, we prove that if
		\begin{equation*}
			\int_{\mathbb{R}^{n}}|u(x,0)|^{2}e^{2\alpha^{p}|x|^p/p}\ \d x
			+\int_{\mathbb{R}^{n}}|u(x,1)|^{2}e^{2\beta^{q}|x|^q/q}\ \d x
			<\infty,
		\end{equation*}
		for some $\alpha,\beta>0$ and there exists $N_{p}>0$ such that
		\begin{equation*}
			\alpha\beta>N_p, 
		\end{equation*}
		then $u\equiv 0$.
		These results can be interpreted as  dynamical versions of the uncertainty principle of Morgan's type.
		Furthermore, as an application, our results extend to a large class of semi-linear Schr\"{o}dinger equations.
	\end{abstract}
	
	\maketitle
	
	\section{Introduction}\label{section1}
	
	\subsection{Background and motivation}
	\
	
	The uncertainty principle, which states that a non-zero function and its Fourier transform cannot be simultaneously sharply localized, is ubiquitous in harmonic analysis. The Fourier transform is given by
	\begin{equation*}
		\hat{f}(\xi):=(2\pi)^{-\frac{n}{2}}\int_{\mathbb{R}^n}f(x)e^{-\i x\cdot\xi}\,\d x.
	\end{equation*}
	In \cite{M34},  Morgan proved the following uncertainty principle in one dimension: Given $1<p\leqslant 2$, if
	\begin{equation}\label{Introduction-3}
		f(x)=O\Big(e^{-\frac{\alpha^p|x|^p}{p}}\Big)\,\,\,\mbox{and}\,\,\, \hat{f}(\xi)=O\Big(e^{-\frac{\beta^q|\xi|^q}{q}}\Big),\,\,\frac{1}{p}+\frac{1}{q}=1,\,\,\,\alpha, \beta>0,
	\end{equation}
	with 
	\begin{equation}\label{alpha-beta}
		\alpha\beta>\Big|\cos{\Big(\frac{p\pi}{2}\Big)}\Big|^{\frac{1}{p}},
	\end{equation}
	then $f\equiv 0$. Moreover, the above constant is sharp. Several related remarks are given as follows:
	\begin{itemize}
		\item [($\textbf{a}_1$)] The above theorem can be viewed as a generalization of the well-known Hardy uncertainty principle \cite{GHH33}, which corresponds to  $p=q=2$ in \eqref{Introduction-3}. In \cite{GS53}, Gel'fand and Shilov introduced the class $Z_p^p$ for $p \geqslant 1 $ as  the space of all functions $\varphi(z_1, \ldots, z_n)$ which are analytic for all values of $ z_1, \ldots, z_n \in \mathbb{C}$ such that
		$$|\varphi(z_1, \ldots, z_n)| \leqslant C_0 e^{\sum_{j=1}^{n} \epsilon_j C_j |z_j|^p},\,\,\,C_j>0,\,\, j = 0, 1, \ldots, n .$$
		Here $\epsilon_j = 1$ for $ z_j$ non-real and $\epsilon_j = -1$ for  $z_j$ real, with $j = 1, \ldots, n$.
		They demonstrated that 
		\begin{equation}\label{Zpp}
			\widehat{Z_p^p}=Z_q^q,
		\end{equation}
		where $\frac{1}{p} + \frac{1}{q} = 1$, i.e., the  Fourier transform of the space $Z_p^p$ is the space $Z_q^q$.
		
		\item [($\textbf{a}_2$)] The following higher dimensional version is established in \cite[Theorem 1.4]{BDJ03}: Let $p\in (1,2)$, $1/p+1/q=1$ and $\alpha, \beta>0$. Assume that $f\in L^2(\R^n)$ and
		\begin{equation}\label{Introduction-5-6}
			\int_{\mathbb{R}^{n}}|f(x)|e^{\alpha^p|x_j|^p/p}\ \d x+\int_{\mathbb{R}^{n}}|\hat{f}(\xi)|e^{\beta^q|\xi_j|^q/q}\ \d\xi<\infty,
		\end{equation}
		for some $j\in \{1,2,\ldots,n\}$. Then $ f\equiv 0$ if \eqref{alpha-beta} holds. Moreover, the constant in \eqref{alpha-beta} is sharp.

		\item [($\textbf{a}_3$)]
		This result  is also related to the following Beurling-H\"{o}rmander uncertainty principle (see \cite{LH91} for $n=1$ and \cite{BDJ03} for higher dimensions): 
		\begin{equation*}\label{Introduction-5}
			f\in L^{2}(\mathbb{R}^{n})\,\,\,\mbox{and}\,\,\,\int_{\mathbb{R}^{n}}\int_{\mathbb{R}^{n}}|f(x)||\hat{f}(\xi)|e^{|x\cdot\xi|}\ \d x \ \d\xi<\infty\ \Longrightarrow \ f\equiv 0.
		\end{equation*}
		Indeed, it implies immediately the following result: 
		\begin{equation}\label{Introduction-6}
			\int_{\mathbb{R}^{n}}|f(x)|e^{\alpha^p|x|^p/p}\ \d x+\int_{\mathbb{R}^{n}}|\hat{f}(\xi)|e^{\beta^q|\xi|^q/q}\ \d\xi<\infty, \ \alpha \beta \geqslant 1\ \Longrightarrow \ f\equiv 0.
		\end{equation}
		It seems that the optimal constant in \eqref{Introduction-6} is still unknown.
		
		\item [($\textbf{a}_4$)]  We refer to \cite{Na93} and the monograph \cite{HJ} for further generalizations.
	\end{itemize}
	In the following, we call \eqref{Introduction-3}, as well its $L^1$-version, \eqref{Introduction-5-6} and \eqref{Introduction-6}, the uncertainty principle of Morgan's type.
	
	There exists an interesting dynamical interpretation of the aforementioned uncertainty principles (see e.g. in \cite{EKPV06}). More precisely, recall that the solution of the free Schr\"{o}dinger equation, with initial datum $u_0\in L^2$, satisfies 
	\begin{align}\label{intro-free}
		u(x,t):=&e^{\i t\Delta}u_{0}(x) \notag\\
		=&(4\pi \i t)^{-n/2}\int_{\mathbb{R}^{n}}e^{\i |x-y|^2/4t}u_{0}(y)\ \d y
		=(2\pi \i t)^{-n/2}e^{\i |x|^2/4t}\widehat{\big(e^{\i |\cdot|^2/4t}u_0\big)}\Big(\frac{x}{2t}\Big).
	\end{align}		
	Thus one can reformulate \eqref{Introduction-6} into the following uniqueness result for the free Schr\"{o}dinger equation: Assume that $T\ne 0$, 
	\begin{equation}\label{Introduction-7}
		\int_{\mathbb{R}^{n}}|u_0(x)|e^{\alpha^p|x|^p/p}\ \d x+\int_{\mathbb{R}^{n}}|u(x,T)|e^{\beta^q|x|^q/q(2T)^q}\ \d x<\infty,  \, \alpha \beta \geqslant 1\, \,   \Longrightarrow \, u_0\equiv 0.
	\end{equation}

	In a series of papers \cite{EKPV06,EKPV08,EKPV08-1,EKPV10,EKPV11},  Escauriaza, Kenig, Ponce and Vega generalized the uniqueness result in this direction for  Schr\"{o}dinger equation with  potentials, namely,
	\begin{equation}\label{Introduction-9}
		\partial_{t}u=\i (\Delta u+V(x,t)u) \,\,\,\, \mbox{in} \,\,\,\mathbb{R}^{n}\times [0,1].
	\end{equation}
	
	For the dynamical versions of Hardy's uncertainty principles, it was proved in \cite{EKPV08-1} that the solutions of \eqref{Introduction-9} (with bounded potentials) satisfy the $L^2$ version of statement \eqref{Introduction-7} with parameters $p=q=2$, $T=1$, and under the stronger assumption $\alpha \beta>2$, specifically,
	\begin{equation*}\label{Introduction-777}
		\int_{\mathbb{R}^{n}}|u_0(x)|^2e^{\alpha^2|x|^2}\ \d x+\int_{\mathbb{R}^{n}}|u(x,1)|^2e^{\beta^2|x|^2/4}\ \d x<\infty,  \, \alpha \beta > 2\, \,   \Longrightarrow \, u_0\equiv 0.
	\end{equation*}
	Subsequently, this result was refined to the optimal condition $\alpha \beta>1$ in \cite{EKPV10}.
	
	In terms of the  dynamical versions of the Morgan type uncertainty principles,  it was proved in \cite[Corollary 1]{EKPV11} that  if $u\in C([0,1];\,\ L^{2}(\mathbb{R}^n))$ is a solution of \eqref{Introduction-9}, where $V=V(x,t)$ is complex-valued, bounded and
	\begin{equation*}
		\lim\limits_{R\to +\infty} \|V\|_{L^{1}([0,1];L^{\infty}(\mathbb{R}^{n}\setminus B_{R}))}=0.
	\end{equation*}
	Let $1<p<2$ and  $1/p+1/q=1$. If 
	\begin{equation}\label{Intro-integral}
		\int_{\mathbb{R}^{n}}|u(x,0)|^2e^{\alpha^p|x|^p/p}\ \d x
		+\int_{\mathbb{R}^{n}}|u(x,1)|^2e^{\beta^q|x|^q/q}\ \d x<\infty,
	\end{equation}
	and  there exists $N_p>0$ such that
	\begin{equation}\label{Introduction-10}
		\alpha \beta>N_p,
	\end{equation}
	then $u\equiv0$.

	
	\textbf{Aim and Motivation.}
	It is quite natural to explore how magnetic potentials affect the unique continuation properties for the equation of the form
	\begin{equation}\label{Introduction-1}
		\i \partial_{t}u+\Delta_{A}u=Vu,
	\end{equation}
	where $V=V(x,t): \mathbb{R}^{n}\times [0,1]\rightarrow \mathbb{C}$ and
	\begin{align*}
		\Delta_{A}:=\nabla_{A}^{2}, \ \ \nabla_{A}:=\nabla-\i A, \ \ A=A(x):\mathbb{R}^{n}\rightarrow \mathbb{R}^{n}.
	\end{align*} 
	Barcel\'{o} et al. \cite{BFGRV13} investigated a dynamical version of the Hardy uncertainty principle associated with equation \eqref{Introduction-1}. More precisely, they proved that if \eqref{Intro-integral}-\eqref{Introduction-10} hold with $p=q=2$ and $N_2=2$, then the solution $u$ vanishes. Later, this result was further improved by Cassano and Fanelli \cite{CF15} under the condition $N_2=1$, which is sharp.
	However, to the best of our knowledge, the general case $1<p<2$ has not been touched upon.
	The goal of this paper is to establish a dynamical version of the uncertainty principle of Morgan type associated with the electromagnetic Schr\"{o}dinger equation  given by \eqref{Introduction-1}.
	
	It is worth mentioning that Escauriaza et al. also proved Hardy's uncertainty principle in the context of  the heat equation with potentials \cite{EKPV16}. Additionally, Fern\'{a}ndez-Bertolin \cite{Fer19} obtained the dynamical version of the Hardy uncertainty principle for the discrete Schr\"{o}dinger equations, see also \cite{Fer17}.
	We refer to the recent survey paper \cite{FM21} for further discussions.

	\subsection{Main results}
	\
	
	
	The assumptions of the potentials $A$ and $V$ in \eqref{Introduction-1} are collected in the following condition.
	
	\textbf{Condition A.} Let $A=A(x)=(A_{1}(x),\ldots,A_{n}(x)):\mathbb{R}^{n}\rightarrow \mathbb{R}^{n}$ be a real magnetic potential and $V=V(x,t)$ be complex-valued.  The magnetic field, denoted by $B\in\mathcal{M}_{n\times n}(\mathbb{R})$, is the anti-symmetric gradient of $A$, i.e.,
	\begin{equation}\label{Introduction-11}
		B=B(x)=DA(x)-DA^{t}(x), \ \ \ B_{jk}(x)=\partial_{x_j}A_{k}(x)-\partial_{x_k}A_{j}(x).
	\end{equation} 
	Assume the following conditions are satisfied:
	
	\noindent \emph{(\romannumeral1) }  The integral 
	\begin{equation}\label{thm1.1-3}
		\int_{0}^{1}A(sx)\ \d s\in \mathbb{R}^{n}
	\end{equation}
	is well defined at almost every $x\in \mathbb{R}^{n}$.
	
	\noindent \emph{(\romannumeral2) } 
	Assume that 
	\begin{equation}\label{M_B}
		\|x^{t}B\|_{L^{\infty}}=:M_B<\infty,   
	\end{equation}
	and for $e_{1}=(1,0,\ldots,0)$, it holds that
	\begin{equation}\label{thm1.1-4}
		e_{1}^{t}B(x)=0.
	\end{equation}
	\noindent \emph{(\romannumeral3) } 
	Define $\Psi(x):=x^{t}B(x)\in \mathbb{R}^{n}$, $\Theta(x):=\int_{0}^{1}\Psi(sx)\ \d s\in \mathbb{R}^{n}$. There exists $\varepsilon'_0, \varepsilon''_0>0$ such that
	\begin{equation}\label{thm1.1-5}
		\|\ |\Theta|^{2}\|_{L^{\infty}}\leqslant \varepsilon'_0,\,\,\, \|\partial_{x_j}\Theta_{j}\|_{L^{\infty}}\leqslant \varepsilon''_0,\,\,\, j=2,\ldots,n.
	\end{equation}
	\noindent \emph{(\romannumeral4) }
	In addition, we assume
	\begin{equation}\label{M_V}
		\|V\|_{L^{\infty}(\mathbb{R}^{n}\times [0,1])}=:M_V<\infty,   
	\end{equation}
	and
	\begin{equation}\label{thm1.1-2}
		\lim\limits_{R\to +\infty} \|V\|_{L^{1}([0,1];L^{\infty}(\mathbb{R}^{n}\setminus B_{R}))}=0.
	\end{equation} 
	We make the following observations concerning to \textbf{Condition A}:
	\begin{itemize}
		
		\item [($\textbf{b}_1$)] It has been proved in \cite[Proposition 2.6]{BFGRV13} that assumptions \eqref{thm1.1-3} and \eqref{M_B} ensure the self-adjointness of $\Delta_{A}$ in $L^2$.

		\item [($\textbf{b}_2$)] Condition \eqref{thm1.1-4}  plays a key role in demonstrating the linear exponential decay estimate (Lemma \ref{Interpolation}) and the Carleman inequality (Lemma \ref{carleman estimate}). These constitute the foundational tools for proving the main results. 
		
		\item [($\textbf{b}_3$)]  Assumptions \eqref{thm1.1-5}--\eqref{thm1.1-2}  impose certain boundedness constraints. It is important to note that these do not necessitate any smallness conditions on the  potentials $A$ and $V$.

	\end{itemize}

	The main result of this paper is as follows.

	\begin{theorem}\label{morgan for p-q}
		Let $n\geqslant 3$ and $1<p<2$. Assume that \textbf{Condition A}  holds. There exists some $N_{p}>0$ such that for any solution  $u\in C([0,1];L^{2}(\mathbb{R}^n))$  of \eqref{Introduction-1} that meets the following criteria for some positive constants $\alpha,\beta>0$:
		\begin{equation}\label{cor1.1-1}
			\int_{\mathbb{R}^{n}}|u(x,0)|^{2}e^{2\alpha^{p}|x|^p/p}\ \d x
			+\int_{\mathbb{R}^{n}}|u(x,1)|^{2}e^{2\beta^{q}|x|^q/q}\ \d x
			<\infty,
		\end{equation}
		where $1/p+1/q=1$. If 
		\begin{equation}\label{cor1.1-2}
			\alpha\beta>N_p, 
		\end{equation}
		then $u\equiv 0$.
	\end{theorem}
	
	As a direct application of Theorem \ref{morgan for p-q}, we obtain the uniqueness of solutions for non-linear equations of the form 
	\begin{equation}\label{Introduction-2}
		\i\partial_{t}u+\Delta_{A}u=F(u,\bar{u}).
	\end{equation}
	
	\begin{corollary}\label{apply morgan}
		Let $n\geqslant 3$ and $1<p<2$. Assume that \textbf{Condition A}  holds and $F:\mathbb{C}^{2}\rightarrow \mathbb{C}, F\in C^{k}$, $F(0)=\partial_{u}F(0)=\partial_{\bar{u}}F(0)=0$, where $k\in \mathbb{Z}^{+}, k>n/2$. There exists some $N_p>0$ such that for any strong solutions $u_{1}, u_{2}\in C\big([0,1]; H^{k}(\mathbb{R}^{n})\big)$ of \eqref{Introduction-2} that meet the following criteria for some positive constants $\alpha,\beta>0$: 
		\begin{equation}\label{thm1.3-1}
			e^{\alpha^{p}|x|^p/p}(u_{1}(0)-u_{2}(0)),\,\, e^{\beta^{q}|x|^q/q}(u_{1}(1)-u_{2}(1))\in L^{2}(\mathbb{R}^{n}),
		\end{equation}
		where $1/p+1/q=1$. If
		\begin{equation}\label{thm1.3-2}
			\alpha\beta>N_p,
		\end{equation}
		then $u_{1}\equiv u_{2}$.
	\end{corollary}
	
	Several remarks on Theorem \ref{morgan for p-q} and Corollary \ref{apply morgan} are as follows:
	\begin{itemize}
		\item [($\textbf{c}_1$)]As far as we are aware, Theorem \ref{morgan for p-q} seems to be the first work to give the connection between unique continuation properties of electromagnetic Schr\"{o}dinger equation \eqref{Introduction-1} and the Morgan uncertainty principle. In the absence of a magnetic field ($A\equiv 0$), Theorem \ref{morgan for p-q} was established in \cite{EKPV11}; whereas for electromagnetic Schr\"{o}dinger equation \eqref{Introduction-1}, uniqueness results of Hardy type  were proved in \cite{BFGRV13,CF15}.
		
		\item [($\textbf{c}_2$)] we do not provide an estimate of the universal constant $N_{p}$. It is worth noting that, even for the free Schr\"{o}dinger equation,  the optimal value of $N_p$ seems unknown in higher dimensions $n\geqslant 2$ (see remark ($\textbf{a}_3$)).
		We also remark that we cannot prove the result in dimension $n=2$. This inability stems from our heavy reliance on condition \eqref{thm1.1-4}; however,  there are no $2\times2$ anti-symmetric matrices with non-trivial kernel.
		
		\item [($\textbf{c}_3$)] We emphasize  that the proof strategy for Theorem \ref{morgan for p-q} is quite different from those employed in the Hardy type results presented  in \cite{BFGRV13,CF15}. Those results fundamentally depend on the logarithmic convexity of the quantity $H(t)=\|e^{g(x, t)}u(t)\|_{L^2}$, where $g$ is some suitable weight function, quadratically growth with respect to $x$. This property ensures that a Gaussian decay observed at times $0$ and $1$ is preserved for intermediate times. However, this phenomenon generally fails when $g(x,t)\approx|x|^p$ for $p\ne 2$. Indeed,  for the free Schr\"{o}dinger evolution, if $u(0)\in Z_p^p$ (see remark ($\textbf{a}_1$)) for some $1<p<2$, then by the formular \eqref{intro-free}, $u(t)\in Z_q^q$ holds for any $t\ne 0$ with $1/p+1/q=1$. Consequently, the solution cannot maintain the decay rate $e^{-c|x|^p}$ uniformly over the interval $0\leqslant t\leqslant 1$.
		Instead, our approach is inspired by the linear exponential decay estimate for the operator $\i\partial_t+\Delta$  established in \cite{KPV03}. We aim to establish a similar upper bound for the magnetic case $\i\partial_t+\Delta_{A}$ (see Lemma \ref{Interpolation}), a result that may hold independent interest. Then we incorporate  ideas from \cite{EKPV11} by combining the lower bounds on solutions with a  Carleman inequality for  $\i\partial_t+\Delta_{A}$ (see Lemma \ref{carleman estimate}).
		
	\end{itemize}

	{\noindent \textbf{Plan of the paper.}} \ The rest of the paper is organized as follows:
	In Section \ref{Preliminaries}, we present some preliminary results; 
	In Section \ref{proof of theorem}, we establish the  linear exponential decay estimates for solutions of the electromagnetic Schr\"{o}dinger equation \eqref{gauge1-0} below and the Carleman inequality for the magnetic operator $\i\partial_t+\Delta_{A}$; In Section \ref{proof of corollary and theorem}, we prove Theorem \ref{morgan for p-q} and Corollary \ref{apply morgan}.

	\section{Preliminaries}\label{Preliminaries}
	
	In this section, we present some tools and  preliminary results that will be used in the proofs of the main results.
	More precisely, consider
	\begin{equation}\label{gauge1-0}
		\partial_{t}u=\i(\Delta_{A}u+V(x,t)u+F(x,t)),
	\end{equation}
	where $A=A(x,t): \mathbb{R}^{n}\times [0,1]\rightarrow \mathbb{R}^{n}$, $V(x,t), F(x,t): \mathbb{R}^{n}\times [0,1]\rightarrow \mathbb{C}$.
	The above equation exhibits gauge invariance in the following sense:  Suppose $u$
	is a solution to \eqref{gauge1-0}. Define $\tilde{A}=A+\nabla\varphi$, where $\varphi=\varphi(x):\mathbb{R}^{n}\rightarrow \mathbb{R}$. Under this transform, the function $\tilde{u}=e^{\i\varphi}u$ is a solution to
	\begin{equation*}
		\partial_{t}\tilde{u}=\i(\Delta_{\tilde{A}}\tilde{u}+V(x,t)\tilde{u}+e^{\i\varphi}F(x,t)).
	\end{equation*}
	\subsection{The Cronstr\"{o}m gauge}\label{Cronst gauge}
	\
	
	\begin{definition}\label{def-crons}
		A connection $\nabla-\i A(x)$ is said to be in the Cronstr\"{o}m gauge (alternatively referred to as  Poincar${\rm \acute{e}}$ gauge or transversal gauge) if the vector potential $A$  is orthogonal to the position vector $x$ for all $x\in \R^n$, i.e., 
		$$A(x)\cdot x=0,\,\,\, \mbox{for all}\,\,\,  x\in \R^n.$$
	\end{definition}
	
	The following result comes from \cite[Lemma 2.2 and Corollary 2.3]{BFGRV13}. 
	\begin{lemma}\label{Cronstrom gauge}
		Let $A=A(x)=(A_1(x),\ldots,A_n(x)):\mathbb{R}^{n}\rightarrow \mathbb{R}^{n}$, for $n\geqslant 2$, and denote $B=DA-DA^{t}\in \mathcal{M}_{n\times n}(\mathbb{R})$, 
		$B_{jk}={\partial_{x_j}A_k}-{\partial_{x_k}A_j}$, and $\Psi(x):=x^{t}B(x)\in \mathbb{R}^{n}$. Assume that the two vector quantities
		\begin{equation}\label{gauge1-1}
			\int_{0}^{1}A(sx)\ \d s\in \mathbb{R}^{n}, \ \ \ \int_{0}^{1}\Psi(sx)\ \d s\in \mathbb{R}^{n}
		\end{equation}
		are finite for almost every $x\in \mathbb{R}^{n}$; moreover, define the function
		\begin{equation}\label{gauge1-2}
			\varphi(x):=-x\cdot\int_{0}^{1}A(sx)\ \d s\in \mathbb{R}.
		\end{equation}
		Then the following two identities hold:
		\begin{align}
			\label{gauge1-3}\tilde{A}(x):=A(x)+\nabla \varphi(x)&=-\int_{0}^{1}\Psi(sx)\ \d s,\\
			\label{gauge1-4}x^{t}D\tilde{A}(x)&=-\Psi(x)+\int_{0}^{1}\Psi(sx)\ \d s.
		\end{align}
		In particular, we have
		\begin{equation}\label{gauge1-5}
			x\cdot \tilde{A}(x)=0, \ \ \ x\cdot x^{t}D\tilde{A}(x)=0.
		\end{equation}
	\end{lemma}
	
	\subsection{The Appell transformation}
	\ 
	
	To simplify the analysis, we consider the scenario where the initial and final states, $u(0)$ and $u(1)$, exhibit identical Gaussian decay properties. This is achieved by applying the following Appell transformation \eqref{Appell1-2},  also known as a pseudo-conformal transformation (see \cite[Lemma 2.7]{BFGRV13}).
	\begin{lemma}\label{Appell trans}
		Let $A=A(y,s)=(A_1(y,s),\ldots,A_n(y,s)):
		\mathbb{R}^{n+1}\rightarrow \mathbb{R}^{n}$, $V=V(y,s)$, $F=F(y,s):\mathbb{R}^{n}\times [0,1]\rightarrow \mathbb{C}$,
		$u=u(y,s):\mathbb{R}^{n}\times [0,1]\rightarrow \mathbb{C}$ be a solution to
		\begin{equation}\label{Appell1-1}
			\partial_{s}u=\i(\Delta_{A}u+V(y,s)u+F(y,s))\,\,\,\, \mbox{in} \,\,\,\mathbb{R}^{n}\times [0,1],
		\end{equation}
		and define, for any $\alpha,\beta>0$, the function
		\begin{equation}\label{Appell1-2}
			\tilde{u}(x,t):=\left(\frac{\sqrt{\alpha\beta}}{\alpha(1-t)+\beta t}\right)^{\frac{n}{2}}
			u\left(\frac{\sqrt{\alpha \beta}x}{\alpha(1-t)+\beta t}, \frac{\beta t}{\alpha(1-t)+\beta t}\right)e^{\frac{(\alpha-\beta)|x|^{2}}{4\i(\alpha(1-t)+\beta t)}}.
		\end{equation}
		Then $\tilde{u}$ is a solution to
		\begin{equation}\label{Appell1-3}
			\partial_{t}\tilde{u}=\i\bigg(\Delta_{\tilde{A}}\tilde{u}+\frac{(\alpha-\beta)\tilde{A}\cdot x}{\alpha(1-t)+\beta t}\tilde{u}+\tilde{V}(x,t)\tilde{u}+\tilde{F}(x,t)\bigg) \,\,\,\, \mbox{in} \,\,\, \mathbb{R}^{n}\times [0,1],
		\end{equation}
		where
		\begin{align}
			\label{Appell1-4}\tilde{A}(x,t)&=\frac{\sqrt{\alpha\beta}}{\alpha(1-t)+\beta t}A\left(\frac{\sqrt{\alpha \beta}x}{\alpha(1-t)+\beta t}, \frac{\beta t}{\alpha(1-t)+\beta t}\right),\\
			\label{Appell1-5}\tilde{V}(x,t)&=\frac{\alpha\beta}{(\alpha(1-t)+\beta t)^{2}}V\left(\frac{\sqrt{\alpha \beta}x}{\alpha(1-t)+\beta t}, \frac{\beta t}{\alpha(1-t)+\beta t}\right),\\
			\label{Appell1-6}\tilde{F}(x,t)&=\left(\frac{\sqrt{\alpha\beta}}{\alpha(1-t)+\beta t}\right)^{\frac{n}{2}+2}F\left(\frac{\sqrt{\alpha \beta}x}{\alpha(1-t)+\beta t}, \frac{\beta t}{\alpha(1-t)+\beta t}\right)e^{\frac{(\alpha-\beta)|x|^{2}}{4\i(\alpha(1-t)+\beta t)}}.
		\end{align}
	\end{lemma}

	\subsection{Weighted estimates for $\nabla_{A}$}\,
	
	
	We first recall the following abstract lemma.
	\begin{lemma}\label{log convex}
		Suppose that $\mathcal{S}$ is a symmetric operator, $\mathcal{A}$ is skew-symmetric, both are allowed to depend smoothly on the time variable; $G$ is a positive function, $f(x,t)$ lie in $C^{\infty}([0,1], \mathscr{S}(\R^n))$,
		\begin{equation*}
			H(t)=(f,f)\footnote{Here $(f,g)=\int_{\R^n}f\bar{g}\ \d x$.},\ \ \ D(t)=(\mathcal{S}f,f),\ \ \ \partial_{t}\mathcal{S}=\mathcal{S}_{t}, \ \ \ N(t)=\frac{D(t)}{H(t)}.
		\end{equation*}
		Then
		\begin{align}\label{convex1-1}
			\partial^{2}_{t}H=&2\partial_{t}{\rm Re}(\partial_{t}f-\mathcal{S}f-\mathcal{A}f,f)+2(\mathcal{S}_{t}f+[\mathcal{S},\mathcal{A}]f,f) \notag\\
			&+\|\partial_{t}f-\mathcal{A}f+\mathcal{S}f\|_{L^{2}}^{2}-\|\partial_{t}f-\mathcal{A}f-\mathcal{S}f\|_{L^{2}}^{2}.
		\end{align}
	\end{lemma}
	\begin{proof}
		See \cite[Lemma 2]{EKPV08-1}.
	\end{proof}
	
	To state the result, we introduce several auxiliary functions. We define 
	\begin{equation}\label{bounded1-5}
		\phi(r):=r\cdot e^{-\int_{0}^{\frac{1}{r}}\frac{e^{-t}-1}{t}\ \d t},\,\,\, r\geqslant 1.
	\end{equation}
	Let $\sigma(x)=|x|$ and define 
	\begin{equation}\label{bounded1-6}
		w(x):=\phi(\sigma(x))=\phi(r),\,\,\, r=|x|\geqslant 1.
	\end{equation}
	For $1<p<2$, we  further  define
	\begin{equation}\label{def-varphi}
		\varphi(x)=\begin{cases}
			|x|^{p}+p(2-p)w(x), & |x|\geqslant 1,\\
			s_1|x|^{2}+s_2, &  |x|<1,
		\end{cases}    
	\end{equation}
	where 
	$$
	s_1=\frac{1}{2}\Big(p+p(2-p)e^{-1}\phi(1)\Big),
	$$
	and
	$$
	s_2=\frac{1}{2}\Big(2-p+p(2-p)(2-e^{-1})\phi(1)\Big).
	$$
	When $r=|x|\geqslant 1$, by the definitions of $\phi(r), w(x), \varphi(x)$, we obtain the following results:
	
	\noindent (i) $\phi(r)$ is an increasing function, and $\phi(r)=O(r)$.
	
	\noindent (ii) $w(x)=O(|x|)$,\,\,$\nabla w(x)=\phi'(r)\frac{x}{r}$ \ and \ $|\nabla w(x)|=O(1)$.
	
	\noindent (iii) $\varphi\geqslant 0$, and 
	\begin{equation}\label{nabla varphi}
		\nabla \varphi(x)=\Big(p|x|^{p-2}+p(2-p)\frac{\phi'(r)}{r}\Big)x, \,\, D^{2}\varphi \geqslant p(p-1)|x|^{p-2}I.    
	\end{equation}

	\noindent In addition, it is easy to verify that $\varphi(x)$ is a strictly convex radial function on $\R^n$, and we have
	\begin{align}\label{bounded1-11}
		\|\partial^{\alpha}\varphi\|_{L^{\infty}(\R^n)}\leqslant c,\ 2\leqslant |\alpha| \leqslant 4;\ \,\,\mbox{and}\,\,\,\|\partial^{\alpha}\varphi\|_{L^{\infty}(|x|\leqslant 2)}\leqslant c,\ |\alpha|\leqslant 4,
	\end{align}
	where $c>0$ is a constant, $\alpha$ denotes a multi-index.
	
	\begin{proposition}\label{bound h}
		Let $v\in C([0,1];L^{2}(\mathbb{R}^{n}))$ be a solution to
		\begin{equation}\label{bounded1-1}
			\partial_{t}v=\i(\Delta_{A}v+Vv)\,\,\,\,\mbox{in}\,\,\,\mathbb{R}^{n}\times [0,1],
		\end{equation}
		$A=A_k=A_k(x,t):\mathbb{R}^{n+1}\rightarrow \mathbb{R}^{n}$, $V=V_k=V_k(x,t):\mathbb{R}^{n+1}\rightarrow \mathbb{C}$. Assume that
		\begin{equation}\label{bounded1-2}
			x\cdot A(x)=x\cdot \partial_{t}A(x)=0.
		\end{equation} 
		Denote $B=B_k=B_k(x,t)=D_{x}A_k-D_{x}A_k^{t}$, where $k$ is a parameter. Suppose that there exists some large and fixed constant $k_0$, such that
		\begin{equation}\label{MBMV}
			\sup_{t\in[0,1]}\|x^{t}B(\cdot,t)\|_{L^{\infty}}\leqslant \Big(\frac{k^m}{a_0}\Big)^{\frac{1}{2p}}M_B,\,\,
			\sup_{t\in[0,1]}\|V(\cdot,t)\|_{L^{\infty}}\leqslant \Big(\frac{k^m}{a_0}\Big)^{\frac{1}{p}}M_{V}, \,\,\, \mbox{if}\,\,\, k\geqslant k_0, 
		\end{equation}
		where $M_B, M_V$ are defined in \eqref{M_B}, \eqref{M_V} respectively, $a_0,m$ are the constants in \eqref{thm1.1-6}-\eqref{thm1.1-7}. 
		
		\noindent For $1<p<2$, assume that
		\begin{equation}\label{bounded1-4}
			\sup_{t\in[0,1]}\|e^{\theta |x|^{p}}v(t)\|_{L^{2}}\leqslant c^{*}k^{c_{p,m}}e^{a_{1}k^{\frac{m}{2-p}}},\,\,\, \theta=(k^m a_0)^{\frac{1}{2}}
		\end{equation}
		holds for $k\geqslant k_{0}$\footnote{Here and in what follows, $c_{p,m}$ denotes a constant that depends only on $p,m,n$; $c^{*}$ represents  a constant that is dependent on $p,m,n,a_0,a_1,a_2,M_B,M_V$, but is independent of $k$. Although the  specific values of these constants may change from line to line, their dependency relationships remain consistent in subsequent discussions.}, where $k_0,a_0,m$ are the constants mentioned above, $a_1$ is the constant in \eqref{thm1.1-7}. Let
		\begin{equation}\label{bounded1-8}
			h(x,t)=e^{\tilde{\theta}\varphi}v(x,t),\,\,\,\tilde{\theta}=\frac{\theta}{2},
		\end{equation}
		where $\varphi$ is defined as in \eqref{def-varphi}. Then we have
		\begin{align}\label{bounded1-9}
			8\tilde{\theta}&\int_{0}^{1}\int_{\mathbb{R}^{n}}t(1-t)\nabla_{A}h\cdot D^{2}\varphi \overline{\nabla_{A}h}\ \d x\ \d t
			+8{\tilde{\theta}}^{3}\int_{0}^{1}\int_{\mathbb{R}^{n}}t(1-t)D^{2}\varphi \nabla \varphi\cdot \nabla \varphi|h|^{2}\ \d x \ \d t\notag\\
			&\leqslant c^{*}k^{c_{p,m}}e^{2a_{1}k^{\frac{m}{2-p}}}.
		\end{align}
	\end{proposition}
	
	\begin{proof}
		
		We first write \eqref{bounded1-1} as the form
		\begin{equation}\label{bounded1-12}
			\partial_{t}v=\i\Delta_{A}v+\i F, \,\,\, F=Vv.
		\end{equation}
		Since $h(x,t)=e^{\tilde{\theta}\varphi}v(x,t)$, it follows that $h$ verifies
		\begin{equation}\label{bounded1-13}
			\partial_{t}h=\mathcal{S}h+\mathcal{A}h+\i e^{\tilde{\theta}\varphi}F \,\,\,\, \mbox{in} \,\,\, \mathbb{R}^{n}\times [0,1],
		\end{equation}
		where symmetric and skew-symmetric operators $\mathcal{S}$ and $\mathcal{A}$ are given as follows
		\begin{equation}\label{bounded1-14}
			\mathcal{S}=-\i\tilde{\theta}(2\nabla \varphi\cdot \nabla_{A}+\Delta \varphi), \,\,\, \mathcal{A}=\i(\Delta_{A}+\tilde{\theta}^{2}|\nabla \varphi|^{2}).
		\end{equation}
		A calculation shows that (see\cite[Lemma 2.9]{BFGRV13})
		\begin{align}\label{bounded1-15}
			(\mathcal{S}_{t}h+[\mathcal{S},\mathcal{A}]h,h)=&4\tilde{\theta}\int_{\mathbb{R}^{n}}\nabla_{A}h\cdot D^{2}\varphi \overline{\nabla_{A}h}\ \d x
			+4\tilde{\theta}^{3}\int_{\mathbb{R}^{n}}D^{2}\varphi \nabla \varphi\cdot \nabla \varphi|h|^{2}\ \d x\notag\\
			&-\tilde{\theta}\int_{\mathbb{R}^n}|h|^{2}\Delta^{2}\varphi\ \d x-2\tilde{\theta}\int_{\mathbb{R}^{n}}|h|^{2}\nabla \varphi\cdot \partial_{t}A\ \d x \notag\\ &-4\tilde{\theta}\ {\rm Im}\int_{\mathbb{R}^{n}}h(\nabla \varphi)^{t}B\cdot \overline{\nabla_{A}h}\ \d x.
		\end{align}
		Notice that from \eqref{convex1-1} it follows
		\begin{align}\label{bounded1-16}
			\partial^{2}_{t}H\geqslant&2\partial_{t}{\rm Re}(\partial_{t}h-\mathcal{S}h-\mathcal{A}h,h)+2(\mathcal{S}_{t}h+[\mathcal{S},\mathcal{A}]h,h)
			-\|\partial_{t}h-\mathcal{A}h-\mathcal{S}h\|_{L^{2}}^{2}.
		\end{align}
		Multiplying \eqref{bounded1-16} by $t(1-t)$ and integrating in $t\in [0,1]$, we obtain
		\begin{equation}\label{bounded1-17}
			2\int_{0}^{1}t(1-t)(\mathcal{S}_{t}h+[\mathcal{S},\mathcal{A}]h,h)\ \d t\leqslant c_{n}\sup_{[0,1]}\|e^{\tilde{\theta}\varphi}v(t)\|_{L^{2}}^{2}
			+c_{n}\sup_{[0,1]}\|e^{\tilde{\theta}\varphi}F(t)\|_{L^{2}}^{2},
		\end{equation}
		where we used $H=(h,h)$, \eqref{bounded1-13}, as well as the H\"{o}lder inequality in above estimate. We remark that the proof of the above inequality can be made rigorous by parabolic regularization (see \cite[proof of Theorem 5]{EKPV08-1}).
		Hence, it follows that
		\begin{align}\label{bounded1-18}
			8\tilde{\theta}&\int_{0}^{1}\int_{\mathbb{R}^{n}}t(1-t)\nabla_{A}h\cdot D^{2}\varphi \overline{\nabla_{A}h}\ \d x\ \d t+8{\tilde{\theta}}^{3}\int_{0}^{1}\int_{\mathbb{R}^{n}}t(1-t)D^{2}\varphi \nabla \varphi\cdot \nabla \varphi|h|^{2}\ \d x\ \d t\notag\\
			&\leqslant c_{n}\sup_{[0,1]}\|e^{\tilde{\theta}\varphi}v(t)\|_{L^{2}}^{2}+c_{n}\|V\|_{L^{\infty}}^{2}\sup_{[0,1]}\|e^{\tilde{\theta}\varphi}v(t)\|_{L^{2}}^{2}
			+c_{n}\tilde{\theta}\sup_{[0,1]}\|e^{\tilde{\theta}\varphi}v(t)\|_{L^{2}}^{2}\notag\\
			&\ \ \ +4\tilde{\theta}\int_{0}^{1}\int_{\mathbb{R}^{n}}t(1-t)|h|^{2}\nabla \varphi\cdot \partial_{t}A\ \d x\ \d t
			+8\tilde{\theta}\ {\rm Im}\int_{0}^{1}\int_{\mathbb{R}^{n}}t(1-t)h(\nabla \varphi)^{t}B\cdot \overline{\nabla_{A}h}\ \d x\ \d t\notag\\
			&\leqslant c^{*}k^{c_{p,m}}e^{2a_{1}k^{\frac{m}{2-p}}}+8\tilde{\theta}\ {\rm Im}\int_{0}^{1}\int_{\mathbb{R}^{n}}t(1-t)h(\nabla \varphi)^{t}B\cdot \overline{\nabla_{A}h}\ \d x\ \d t, \,\,\,\,\,\mbox{if}\,\,\,\, k\geqslant k_0,
		\end{align}
		where in the first inequality, we used \eqref{bounded1-11},\eqref{bounded1-12}, \eqref{bounded1-15} and \eqref{bounded1-17}; while in the second inequality, we used \eqref{nabla varphi},\eqref{bounded1-2}-\eqref{bounded1-8}. The constants $k,k_0,m,c^*,c_{p,m}$ as in \eqref{MBMV}-\eqref{bounded1-4}.
		
		Next, it remains to bound the last term on the right-hand side of \eqref{bounded1-18}. By Young’s inequality with exponents $(\frac{1}{\varepsilon},\varepsilon)$ where $\varepsilon$ to be determined later (see \eqref{varepsilon}), we obtain
		\begin{align}\label{bounded1-19}
			8&\tilde{\theta}\ {\rm Im}\int_{0}^{1}\int_{\mathbb{R}^{n}}t(1-t)h(\nabla \varphi)^{t}B\cdot \overline{\nabla_{A}h}\ \d x\ \d t\notag\\
			&=8\tilde{\theta}\ {\rm Im} \int_{0}^{1}\int_{B_1} t(1-t)h(\nabla \varphi)^{t}B\cdot \overline{\nabla_{A}h}\ \d x\ \d t 
			+8\tilde{\theta}\ {\rm Im} \int_{0}^{1}\int_{B_{1}^{c}}t(1-t)h(\nabla \varphi)^{t}B\cdot \overline{\nabla_{A}h}\ \d x\ \d t \notag\\
			&\leqslant \frac{4}{\varepsilon}\tilde{\theta} \int_{0}^{1}\int_{B_{1}}t(1-t)|h|^{2}\|(\nabla \varphi)^{t}B\|_{L^{\infty}}^{\frac{2}{3}}\ \d x\ \d t
			+4\varepsilon\tilde{\theta} \int_{0}^{1}\int_{B_{1}}t(1-t)\|(\nabla \varphi)^{t}B\|_{L^{\infty}}^{\frac{4}{3}}|\nabla_{A}h|^{2}\ \d x\ \d t  \notag\\
			&\ \ \ +\frac{4}{\varepsilon}\tilde{\theta} \int_{0}^{1}\int_{B_{1}^{c}}t(1-t)|h|^{2}\|(\nabla \varphi)^{t}B\|_{L^{\infty}}^{\frac{2}{3}}\ \d x\ \d t
			+4\varepsilon\tilde{\theta}\int_{0}^{1}\int_{B_{1}^{c}}t(1-t)\|(\nabla \varphi)^{t}B\|_{L^{\infty}}^{\frac{4}{3}}|\nabla_{A}h|^{2}\ \d x\ \d t\notag\\
			&=:I+II+III+IV.
		\end{align}
		Here and in what follows, $B_1$ denotes the open unit ball in $\R^n$, while $B_1^c=\R^n \backslash B_1$.
		
		Now we estimate each term of the above inequality.
		
		For the term $I, II, III$, we obtain that for $k\geqslant k_0$,
		\begin{align}
			\label{bounded1-20} I&\leqslant \frac{c_{p,n}}{\varepsilon}\|x^{t}B\|_{L^{\infty}}^{\frac{2}{3}}\tilde{\theta}\int_{0}^{1}\int_{B_{1}}t(1-t)|h|^{2}\ \d x\ \d t
			\leqslant  \frac{1}{\varepsilon}c^{*}k^{c_{p,m}}e^{2a_{1}k^{\frac{m}{2-p}}},\\
			\label{bounded1-21} II&\leqslant c_{p,n}\varepsilon\|x^{t}B\|_{L^{\infty}}^{\frac{4}{3}}\tilde{\theta}\int_{0}^{1}\int_{B_{1}}t(1-t)|\nabla_{A}h|^{2}\ \d x\ \d t
			\leqslant \varepsilon c^{*}k^{c_{p,m}}e^{2a_{1}k^{\frac{m}{2-p}}},\\
			\label{bounded1-22} III&\leqslant \frac{c_{p,n}}{\varepsilon}\|x^{t}B\|_{L^{\infty}}^{\frac{2}{3}}
			\sup_{|x|>1}\Big(p|x|^{p-2}+p(2-p) \frac{\phi'(|x|)}{|x|}\Big)^{\frac{2}{3}}\tilde{\theta}\int_{0}^{1}\int_{B_{1}^{c}}t(1-t)|h|^{2}\ \d x\ \d t \notag\\
			&\leqslant \frac{c_{p,n}}{\varepsilon}\tilde{\theta}\int_{0}^{1}\int_{B_{1}^{c}}t(1-t)|h|^{2}\ \d x\ \d t
			\leqslant  \frac{1}{\varepsilon}c^{*}k^{c_{p,m}}e^{2a_{1}k^{\frac{m}{2-p}}},
		\end{align}  
		where $k, k_0,m, c^*,c_{p,m}$ as in \eqref{MBMV}-\eqref{bounded1-4}. In the first two estimates, we used \eqref{def-varphi}, \eqref{MBMV}-\eqref{bounded1-8}; in the third estimate, we used \eqref{nabla varphi}, \eqref{MBMV}-\eqref{bounded1-8}.
		
		For the term $IV$, using \eqref{nabla varphi}, we have
		\begin{equation}\label{bounded1-23}
			IV\leqslant 4\varepsilon \tilde{\theta}\|x^{t}B\|_{L^{\infty}}^{\frac{4}{3}}
			\int_{0}^{1}\int_{B_{1}^{c}}t(1-t)\Big(p|x|^{p-2}+p(2-p) \frac{\phi'(|x|)}{|x|}\Big)^{\frac{4}{3}}|\nabla_{A}h|^{2}\ \d x\ \d t.
		\end{equation}
		To further estimate the right hand side of $\eqref{bounded1-23}$, we multiply \eqref{bounded1-16} by $t(1-t)$ and integrate in $t\in [0,1]$, this yields
		\begin{align*}
			\int_{0}^{1}t(1-t)\partial^{2}_{t}H\ \d t&\geqslant2\int_{0}^{1}t(1-t)\partial_{t}{\rm Re}(\partial_{t}h-\mathcal{S}h-\mathcal{A}h,h)\ \d t+2\int_{0}^{1}t(1-t)(\mathcal{S}_{t}h+[\mathcal{S},\mathcal{A}]h,h)\ \d t\\
			&\ \ \ -\int_{0}^{1}t(1-t)\|\partial_{t}h-\mathcal{A}h-\mathcal{S}h\|_{L^{2}}^{2}\ \d t.
		\end{align*}
		We examine each term in the aforementioned inequality separately.  For the first term $\partial_{t}^{2}H$, integrating twice by parts we get
		\begin{align}\label{bounded1-24}
			\int_{0}^{1}t(1-t)\partial_{t}^{2}H(t)\ \d t=H(1)+H(0)-2\int_{0}^{1}H(t)\ \d t\leqslant 2\sup_{[0,1]}\|h(\cdot,t)\|_{L^{2}}^{2}.
		\end{align}
		For the second term, integrating by parts and applying Schwartz inequality, we obtain
		\begin{align}\label{bounded1-25}
			2\int_{0}^{1}\int_{B_{1}^{c}}t(1-t)\partial_{t}{\rm Re}\ \overline{h}(\partial_{t}-\mathcal{S}-\mathcal{A})h\ \d x\ \d t 
			&=-2\int_{0}^{1}\int_{B_{1}^{c}}(1-2t){\rm Re}\ \overline{h}(\partial_{t}-\mathcal{S}-\mathcal{A})h\ \d x\ \d t \notag\\
			&\geqslant -\Big(\sup_{[0,1]}\|\partial_{t}h-\mathcal{S}h-\mathcal{A}h\|_{L^{2}}^{2}+\sup_{[0,1]}\|h(\cdot,t)\|_{L^{2}}^{2}\Big)\notag\\
			&\geqslant -(\|V\|_{L^{\infty}}^{2}+1)\sup_{[0,1]}\|h(\cdot,t)\|_{L^{2}}^{2}.
		\end{align}
		For the third term, we have
		\begin{align}\label{bounded1-26}
			2&\int_{0}^{1}\int_{B_{1}^{c}}t(1-t)\overline{h}(\mathcal{S}_{t}+[\mathcal{S},\mathcal{A}])h \ \d x\ \d t\notag\\
			&\geqslant 8\tilde{\theta}\int_{0}^{1}\int_{B_{1}^{c}}t(1-t)\nabla_{A}h\cdot D^{2}\varphi \overline{\nabla_{A}h}\ \d x\ \d t-2\tilde{\theta}\int_{0}^{1}\int_{B_{1}^{c}}t(1-t)|h|^{2}\Delta^{2}\varphi\ \d x\ \d t\notag\\ 
			&\ \ \ -8\tilde{\theta}\ {\rm Im}\int_{0}^{1}\int_{B_{1}^{c}}t(1-t)h(\nabla \varphi)^{t}B\cdot \overline{\nabla_{A}h}\ \d x\ \d t\notag\\
			&\geqslant -c_{p,n}(\frac{1}{\varepsilon}+1)\tilde{\theta}\sup_{[0,1]}\|h(\cdot,t)\|_{L^{2}}^{2} +\int_{0}^{1}\int_{B_{1}^{c}}t(1-t)\bigg[8\tilde{\theta}p(p-1)|x|^{p-2}-4\varepsilon\tilde{\theta}\|x^{t}B\|_{L^{\infty}}^{\frac{4}{3}}\notag\\
			&\hspace*{5 cm}\cdot\Big(p|x|^{p-2}+p(2-p)\frac{\phi'(|x|)}{|x|}\Big)^{\frac{4}{3}}\bigg]|\nabla_{A}h|^{2}\ \d x\ \d t,
		\end{align}
		where in the first inequality, we used \eqref{nabla varphi},\eqref{bounded1-2}, \eqref{bounded1-15}; while in the second inequality, we used \eqref{nabla varphi}, \eqref{bounded1-11}, and Schwartz inequality. 
		
		\noindent While for the last term $\|\partial_{t}h-\mathcal{A}h-\mathcal{S}h\|_{L^{2}}^{2}$,  we derive from \eqref{bounded1-13} that
		\begin{align}\label{bounded1-27}
			-\int_{0}^{1}t(1-t)\|\partial_{t}h-\mathcal{S}h-\mathcal{A}h\|_{L^{2}(B_{1}^{c})}^{2}\ \d t
			&\geqslant -\sup_{[0,1]}\|\partial_{t}h-\mathcal{S}h-\mathcal{A}h\|_{L^{2}}^{2}\int_{0}^{1}t(1-t)\ \d t\notag\\
			&\geqslant -\frac{1}{6}\|V\|_{L^{\infty}}^{2}\sup_{[0,1]}\|h(\cdot,t)\|_{L^{2}}^{2} .
		\end{align}
		Combining \eqref{bounded1-24}-\eqref{bounded1-27}, we have
		\begin{align}\label{bounded1-28}
			\int_{0}^{1}&\int_{B_{1}^{c}}t(1-t)\bigg[8\tilde{\theta}p(p-1)|x|^{p-2}-4\varepsilon\tilde{\theta}\|x^{t}B\|_{L^{\infty}}^{\frac{4}{3}}
			\Big(p|x|^{p-2}+p(2-p)\frac{\phi'(|x|)}{|x|}\Big)^{\frac{4}{3}}\bigg]|\nabla_{A}h|^{2}\ \d x\ \d t\notag\\
			&\leqslant c_{p,n}(\|V\|_{L^{\infty}}^{2}+\frac{1}{\varepsilon}+1)\tilde{\theta}\sup_{[0,1]}\|h(\cdot,t)\|_{L^{2}}^{2}.
		\end{align}
		We choose
		\begin{equation}\label{varepsilon}
			\varepsilon=2^{-\frac{4}{3}}p^{-\frac{1}{3}}(p-1)a_{0}^{\frac{4}{3p}}k^{-\frac{4}{3p}}M_{B}^{-\frac{4}{3}}\Big(1+(2-p)^{\frac{4}{3}}\phi(1)^{\frac{4}{3}}\Big)^{-1},    
		\end{equation}
		and from \eqref{bounded1-23} and \eqref{bounded1-28}, we obtain
		\begin{align}\label{bounded1-29}
			IV \leqslant c^{*}k^{c_{p,m}}e^{2a_{1}k^{\frac{m}{2-p}}}, \,\,\,\, \mbox{if} \,\,\, k\geqslant k_0,
		\end{align}
		where $k,k_0,m,c^*,c_{p,m}$ as in \eqref{MBMV}-\eqref{bounded1-4}.
		By combining \eqref{bounded1-18}  through \eqref{bounded1-22} with \eqref{bounded1-29} and the  chosen value of $\varepsilon$, we achieve  the desired inequality \eqref{bounded1-9}. The proof of Proposition \ref{bound h} is complete.
	\end{proof}
	
	\section{Linear exponential decay and Carleman inequality}\label{proof of theorem}
	
	We begin by establishing the a linear exponential decay estimate for solutions of the electromagnetic Schr\"{o}dinger equation \eqref{eq-25-2-3}, a result that may be of independent interest.  We  adapt ideas from \cite{KPV03}, with the key step involving energy estimates on the projection of the solution onto both the positive and negative
	frequencies. Additionally, we utilize Calder\'{o}n's first commutator estimate to facilitate the analysis.
	\begin{lemma}\label{Interpolation}
		Suppose that there exists $\varepsilon_{0}>0$ such that $\mathbb{V}: \mathbb{R}^{n}\times[0,1]\rightarrow \mathbb{C}$ satisfies
		\begin{equation*}
			\|\mathbb{V}\|_{L^{1}_{t}L^{\infty}_{x}} \leqslant \varepsilon_{0},
		\end{equation*}
		and $\mathbb{A}=\mathbb{A}(x,t): \mathbb{R}^{n+1}\rightarrow \mathbb{R}^{n}$ with $e_{1}\cdot \mathbb{A}(x,t)=0$ satisfies
		\begin{align*}
			\| \ |\mathbb{A}|^2\|_{L^{1}_{t}L^{\infty}_{x}} \leqslant \varepsilon_{0}', \,\, \,\,\,\mbox{and}\,\,\,\,\,\|\partial_{x_{j}}\mathbb{A}_j\|_{L^{1}_{t}L^{\infty}_{x}}\leqslant \varepsilon_{0}'', \,\, j=2,\ldots,n,
		\end{align*}
		where $\varepsilon_{0}', \varepsilon_{0}''$ are given by \eqref{thm1.1-5}. Let  $u\in C([0,1]; L_{x}^{2}(\mathbb{R}^n))$ be a strong solution of the  perturbed equation
		\begin{align}\label{eq-25-2-3}
			\begin{cases}
				\i\partial_{t}u+\Delta_{\mathbb{A}}u=\mathbb{V}(x,t)u+\mathbb{F}(x,t),& (x,t)\in \mathbb{R}^{n}\times [0,1],\\
				u(x,0)=u_{0}(x),
			\end{cases}
		\end{align}
		with $\mathbb{F}\in L^{1}_{t}([0,1]; L^{2}_{x}(\mathbb{R}^{n}))$. If for some $\lambda\in \mathbb{R}^{n}$, 
		\begin{equation*}
			u_0,\, u_1=u(\cdot,1)\in L^2(e^{2\lambda\cdot x}\ \d x),\,\, \,\,\,\mbox{and}\,\,\,\,\, \mathbb{F}\in L^{1}_{t}([0,1]; L^2(e^{2\lambda\cdot x}\ \d x)),
		\end{equation*}
		then there exists a constant $C>0$ independent of $\lambda$ such that
		\begin{equation}\label{equ lemma1-0}
			\sup_{0\leqslant t \leqslant 1}\|e^{\lambda\cdot x}u(\cdot,t)\|_{L^{2}} \leqslant C\Big(\|e^{\lambda\cdot x}u_0\|_{L^{2}}+\|e^{\lambda\cdot x}u_1\|_{L^{2}}
			+\int_{0}^{1}\|e^{\lambda\cdot x}\mathbb{F}(\cdot,t)\|_{L^{2}}\ \d t\Big).
		\end{equation}
	\end{lemma}
	
	\begin{proof}
		By a standard coordinate rotation, it suffices to show that for some $\beta \in \mathbb{R}$,
		\begin{equation*}
			u_0,\, u_1=u(\cdot,1)\in L^2(e^{2\beta x_1}\ \d x),\,\, \,\,\,\mbox{and}\,\,\,\,\, \mathbb{F}\in L^{1}_{t}([0,1]; L^2(e^{2\beta x_1}\ \d x)),
		\end{equation*}
		there exists a constant $C>0$ independent of $\beta$ such that
		\begin{equation}\label{equ lemma1-1}
			\sup_{0\leqslant t \leqslant 1}\|e^{\beta x_1}u(\cdot,t)\|_{L^{2}} \leqslant C\Big(\|e^{\beta x_1}u_0\|_{L^{2}}+\|e^{\beta x_1}u_1\|_{L^{2}}
			+\int_{0}^{1}\|e^{\beta x_1}\mathbb{F}(\cdot,t)\|_{L^{2}}\ \d t\Big).
		\end{equation}
		
		Without loss of generality, we assume that $\beta>0$.
		
		\textbf{Step 1: Regularization and change of variables.}
		
		We define $\varphi_{n}\in C^{\infty}(\mathbb{R})$, such that $0\leqslant \varphi_{n}\leqslant 1$, and
		\begin{align*}
			\varphi_{n}(s)=
			\begin{cases}
				1,&  s\leqslant n,\\
				0,&  s>10n,
			\end{cases}
		\end{align*}
		with
		$$
		|\varphi_{n}^{(k)}(s)|\leqslant \frac{c_k}{n^k},\,\, k=0,1,\ldots.
		$$
		Using $\varphi_{n}$, we define $\theta_{n}\in C^{\infty}(\mathbb{R})$ by
		\begin{align*}
			\theta_{n}(s):=\beta \int_{0}^{s}\varphi_{n}^{2}(\ell)\ d\ell,
		\end{align*}
		which satisfies
		\begin{align*}
			\theta_{n}(s)=
			\begin{cases}
				\beta s,& s\leqslant n,\\
				c_n\beta,& s>10n,
			\end{cases}
		\end{align*}
		and $\theta_{n}'(s)=\beta \varphi_{n}^{2}(s)\leqslant \beta$, with $|\theta_{n}^{(k)}(s)|\leqslant \frac{\beta c_k}{n^{k-1}}$, $k=1, 2,\ldots.$
		Finally, we define the function $\phi_{n}(s)=e^{\theta_{n}(s)}$, which satisfies $\phi_{n}(s)\leqslant e^{\beta s}$ and $\phi_{n}(s)\rightarrow e^{\beta s}$ as $n\rightarrow \infty$.
		
		We shall write the equation for $v_{n}(x,t)=\phi_{n}(x_1)u(x,t)$. 
		Observe the following identities:
		\begin{align*}
			\phi_{n}\partial_{t}u&=\partial_{t}v_{n},\\
			\phi_{n}\partial_{x_1}u&=\partial_{x_1}v_{n}-\beta \varphi_{n}^{2}(x_1)v_{n},\\
			\phi_{n}\Delta_{\mathbb{A}}u&=\Delta_{\mathbb{A}}v_{n}-2\beta \varphi_{n}^{2}\partial_{x_1}v_{n}-(2\beta \varphi_{n}\varphi'_{n}-\beta^{2}\varphi_{n}^{4}-2\beta \i \varphi_{n}^{2}\mathbb{A}_{1})v_{n}.
		\end{align*}
		From these, we derive
		\begin{equation*}
			\i\partial_{t}v_{n}+\Delta_{\mathbb{A}}v_{n}=\mathbb{V}v_{n}+2\beta \varphi_{n}^{2}\partial_{x_1}v_{n}+(2\beta \varphi_{n}\varphi'_{n}-\beta^{2}\varphi_{n}^{4}-2\beta \i \varphi_{n}^{2}\mathbb{A}_{1})v_{n}+\phi_{n}(x_1)\mathbb{F}.
		\end{equation*}
		Notice that $e_{1}\cdot \mathbb{A}(x,t)=0$, it follows that  $v_{n}(x,t)$ satisfies the following equation
		\begin{equation}\label{equ lemma1-2}
			\i\partial_{t}v_{n}+\Delta_{\mathbb{A}}v_{n}=\mathbb{V}v_{n}+2\beta \varphi_{n}^{2}\partial_{x_1}v_{n}+(2\beta \varphi_{n}\varphi'_{n}-\beta^{2}\varphi_{n}^{4})v_{n}+\phi_{n}(x_1)\mathbb{F}.
		\end{equation}
		
		To eliminate the term $\beta^{2}\varphi_{n}^{4}$,  
		we introduce a new function
		\begin{equation}\label{equ lemma1-3}
			w_{n}(x,t)=e^{-\i\beta^{2}\varphi_{n}^{4}(x_1)t}v_{n}(x,t)=:e^{\mu}v_{n}(x,t).
		\end{equation}
		We now seek a differential equation  satisfied by  $w_{n}$.
		Denote
		\begin{align*}
			&\tilde{\mathbb{F}}_{n}(x,t)=e^{\mu}\phi_{n}(x_1)\mathbb{F}(x,t), \quad\,  \tilde{\mathbb{F}}(x,t)=e^{\mu}e^{\beta x_1}\mathbb{F}(x,t),\\
			&a^{2}(x_1)=2\beta \varphi_{n}^{2}(x_1), \,\qquad\qquad\,   b(x_1)=-8\beta^{2}\varphi_{n}^{3}(x_1)\varphi'_{n}(x_1),\\
			&h(x_1,t)=-(4\i\beta^{2}\varphi_{n}^{3}\varphi'_{n}t)^{2}-12\i\beta^{2}\varphi_{n}^{2}(\varphi'_{n})^{2}t-4\i\beta^{2}\varphi_{n}^{3}\varphi''_{n}t+2\beta \varphi_{n}\varphi'_{n}+8\i\beta^{3}\varphi_{n}^{5}\varphi'_{n}t.
		\end{align*}
		Following a calculation analogous to that for $v_{n}(x,t)$, we transform equation \eqref{equ lemma1-2}  into the following differential equation for  $w_{n}$:
		\begin{equation}\label{equ lemma1-4}
			\i\partial_{t}w_{n}+\Delta_{\mathbb{A}}w_{n}=\mathbb{V}w_{n}+\tilde{\mathbb{F}}_{n}+hw_{n}+ a^{2}(x_1)\partial_{x_1}w_n+\i tb(x_1)\partial_{x_1}w_n.
		\end{equation}
		We observe the following decay properties of the coefficients in \eqref{equ lemma1-4}, that is,
		\begin{align}
			\label{equ lemma1-5} \|\partial_{x_1}^{k}h(x_1,t)\|_{L^{\infty}(\mathbb{R}\times [0,1])}\leqslant& \frac{c_k}{n^{k+1}}, \ k=0,1,\ldots, \\
			\label{equ lemma1-6} \|\partial_{x_1}^{k}a^{2}(x_1)\|_{L^{\infty}(\mathbb{R})}\leqslant& \frac{c_k}{n^{k}},  \ a^{2}(x_1)\geqslant 0, \ k=0,1,\ldots, \\
			\label{equ lemma1-7} \|\partial_{x_1}^{k}b(x_1)\|_{L^{\infty}(\mathbb{R})}\leqslant& \frac{c_k}{n^{k}}, \ b(x_1) \ \mbox{real},\ k=0,1,\ldots.
		\end{align}
		
		\textbf{Step 2: Projection estimates.}
		
		
		First, we define $\eta \in C^{\infty}_{0}(\mathbb{R}^{n})$ such that $0\leqslant \eta(x) \leqslant 1$, and
		\begin{align*}
			\eta(x)=
			\begin{cases}
				1,& |x|\leqslant \frac{1}{2},\\
				0,& |x|\geqslant 1,
			\end{cases}
		\end{align*}
		Additionally, we define $\chi_{\pm}(\xi)$ as
		\begin{align*}
			\chi_{\pm}(\xi)=
			\begin{cases}
				1,& \xi_1>0 \ (\xi_1<0),\\
				0,& \xi_1<0 \ (\xi_1>0).
			\end{cases}
		\end{align*}
		For $\varepsilon \in (0,1]$, we introduce two projections:
		\begin{align}\label{projection}
			\widehat{P_{\varepsilon}f}(\xi):=\eta(\varepsilon \xi)\hat{f}(\xi),\,\,\, 
			\widehat{P_{\pm}f}(\xi):=\chi_{\pm}(\xi)\hat{f}(\xi).
		\end{align}
		In this step, we prove the following
		
		\textbf{Claim}. We have  the following inequality:
		\begin{align}\label{equ lemma1-27}
			\partial_{t}\int_{\mathbb{R}^{n}}|P_{\varepsilon}P_{+}w_{n}|^2\ \d x
			\leqslant& 6c\sum_{j=2}^{n}\|\partial_{x_{j}}\mathbb{A}_{j}\|_{L^{\infty}}\|w_n\|_{L^{2}}^{2}
			+2c\|\ |\mathbb{A}|^{2}\|_{L^{\infty}}\left\|w_{n}\right\|_{L^{2}}^{2} \notag\\
			\ \ &+2c\left\|\mathbb{V}\right\|_{L^{\infty}}\left\|w_{n}\right\|_{L^{2}}^{2}+2c\|\tilde{\mathbb{F}}\|_{L^{2}}\left\|w_{n}\right\|_{L^{2}}
			+\frac{c}{n}\left\|w_{n}\right\|_{L^{2}}^{2},
		\end{align}
		where $c>0$ is a constant independent of $\varepsilon\in (0,1]$ and $n\in \mathbb{Z}^{+}$.
		
		To establish this inequality, we apply the projections to each term in \eqref{equ lemma1-4}, yielding
		\begin{align}\label{equ lemma1-8}
			\i&\partial_{t}P_{\varepsilon}P_{+}w_{n}+\Delta P_{\varepsilon}P_{+}w_{n}-P_{\varepsilon}P_{+}(2\i\mathbb{A}\cdot \nabla w_{n})
			-P_{\varepsilon}P_{+}(\i(\operatorname{div} \mathbb{A})w_{n})-P_{\varepsilon}P_{+}(|\mathbb{A}|^{2}w_{n})\notag\\
			&=P_{\varepsilon}P_{+}(\mathbb{V}w_{n})+P_{\varepsilon}P_{+}(\tilde{\mathbb{F}}_{n})+P_{\varepsilon}P_{+}(hw_{n})+P_{\varepsilon}P_{+}(a^{2}\partial_{x_1}w_{n})
			+P_{\varepsilon}P_{+}(\i tb\partial_{x_1}w_{n}).
		\end{align}
		Taking the complex conjugate of \eqref{equ lemma1-8}, we obtain
		\begin{align}\label{equ lemma1-9}
			-&\i\partial_{t}\overline{P_{\varepsilon}P_{+}w_{n}}+\Delta \overline{P_{\varepsilon}P_{+}w_{n}}-\overline{P_{\varepsilon}P_{+}(2\i\mathbb{A}\cdot \nabla w_{n})}
			-\overline{P_{\varepsilon}P_{+}(\i(\operatorname{div} \mathbb{A})w_{n})}-\overline{P_{\varepsilon}P_{+}(|\mathbb{A}|^{2}w_{n})}\notag\\
			&=\overline{P_{\varepsilon}P_{+}(\mathbb{V}w_{n})}+\overline{P_{\varepsilon}P_{+}(\tilde{\mathbb{F}}_{n})}+\overline{P_{\varepsilon}P_{+}(hw_{n})}
			+\overline{P_{\varepsilon}P_{+}(a^{2}\partial_{x_1}w_{n})}+\overline{P_{\varepsilon}P_{+}(\i tb\partial_{x_1}w_{n})}.
		\end{align}
		Multiplying $\eqref{equ lemma1-8}$ by $\overline{P_{\varepsilon}P_{+}w_{n}}$ and $\eqref{equ lemma1-9}$ by $P_{\varepsilon}P_{+}w_{n}$, and then subtracting the latter from the former, we derive
		\begin{align*}
			\i\partial_{t}|P_{\varepsilon}&P_{+}w_{n}|^2+\Delta P_{\varepsilon}P_{+}w_{n}\cdot \overline{P_{\varepsilon}P_{+}w_{n}}
			-\overline{\Delta P_{\varepsilon}P_{+}w_{n}}\cdot P_{\varepsilon}P_{+}w_{n}\notag\\
			&-P_{\varepsilon}P_{+}(2\i\mathbb{A}\cdot \nabla w_{n})\cdot \overline{P_{\varepsilon}P_{+}w_{n}}+\overline{P_{\varepsilon}P_{+}(2\i\mathbb{A}\cdot \nabla w_{n})}\cdot P_{\varepsilon}P_{+}w_{n}\notag\\
			&-P_{\varepsilon}P_{+}(\i(\operatorname{div} \mathbb{A})w_{n})\cdot \overline{P_{\varepsilon}P_{+}w_{n}}+\overline{P_{\varepsilon}P_{+}(\i(\operatorname{div} \mathbb{A})w_{n})}\cdot P_{\varepsilon}P_{+}w_{n}\notag\\
			&-P_{\varepsilon}P_{+}(|\mathbb{A}|^{2}w_{n})\cdot \overline{P_{\varepsilon}P_{+}w_{n}}+\overline{P_{\varepsilon}P_{+}(|\mathbb{A}|^{2}w_{n})}\cdot P_{\varepsilon}P_{+}w_{n}\notag\\
			=&P_{\varepsilon}P_{+}(\mathbb{V}w_{n})\cdot \overline{P_{\varepsilon}P_{+}w_{n}}-\overline{P_{\varepsilon}P_{+}(\mathbb{V}w_{n})}\cdot P_{\varepsilon}P_{+}w_{n}\notag\\
			&+P_{\varepsilon}P_{+}(\tilde{\mathbb{F}}_{n})\cdot \overline{P_{\varepsilon}P_{+}w_{n}}-\overline{P_{\varepsilon}P_{+}(\tilde{\mathbb{F}}_{n})}\cdot P_{\varepsilon}P_{+}w_{n}\notag\\
			&+P_{\varepsilon}P_{+}(hw_{n})\cdot \overline{P_{\varepsilon}P_{+}w_{n}}-\overline{P_{\varepsilon}P_{+}(hw_{n})}\cdot P_{\varepsilon}P_{+}w_{n}\notag\\
			&+P_{\varepsilon}P_{+}(a^{2}\partial_{x_1}w_{n})\cdot \overline{P_{\varepsilon}P_{+}w_{n}}-\overline{P_{\varepsilon}P_{+}(a^{2}\partial_{x_1}w_{n})}\cdot P_{\varepsilon}P_{+}w_{n}\notag\\
			&+P_{\varepsilon}P_{+}(\i tb\partial_{x_1}w_{n})\cdot \overline{P_{\varepsilon}P_{+}w_{n}}-\overline{P_{\varepsilon}P_{+}(\i tb\partial_{x_1}w_{n})}\cdot P_{\varepsilon}P_{+}w_{n}.
		\end{align*}
		Taking the imaginary part of the above equation, we have
		\begin{align}\label{equ lemma1-11}
			\partial_{t}|P_{\varepsilon}&P_{+}w_{n}|^2+2 {\rm Im}(\Delta P_{\varepsilon}P_{+}w_{n}\cdot \overline{P_{\varepsilon}P_{+}w_{n}})\notag\\
			=&4{\rm Re}(P_{\varepsilon}P_{+}(\mathbb{A}\cdot \nabla w_{n})\cdot \overline{P_{\varepsilon}P_{+}w_{n}})
			+2{\rm Re}(P_{\varepsilon}P_{+}((\operatorname{div} \mathbb{A})w_{n})\cdot \overline{P_{\varepsilon}P_{+}w_{n}})\notag\\
			&+2{\rm Im}(P_{\varepsilon}P_{+}(|\mathbb{A}|^{2}w_{n})\cdot \overline{P_{\varepsilon}P_{+}w_{n}})
			+2{\rm Im}(P_{\varepsilon}P_{+}(\mathbb{V}w_{n})\cdot \overline{P_{\varepsilon}P_{+}w_{n}})\notag\\
			&+2{\rm Im}(P_{\varepsilon}P_{+}(\tilde{\mathbb{F}}_{n})\cdot \overline{P_{\varepsilon}P_{+}w_{n}})
			+2{\rm Im}(P_{\varepsilon}P_{+}(hw_{n})\cdot \overline{P_{\varepsilon}P_{+}w_{n}})\notag\\
			&+2{\rm Im}(P_{\varepsilon}P_{+}(a^{2}\partial_{x_1}w_{n})\cdot \overline{P_{\varepsilon}P_{+}w_{n}})
			+2{\rm Re}(P_{\varepsilon}P_{+}(tb\partial_{x_1}w_{n})\cdot \overline{P_{\varepsilon}P_{+}w_{n}}).
		\end{align}
		
		Now we integrate on both sides of \eqref{equ lemma1-11} and estimate each term in the integration. Let
		\begin{align*}
			W_1:=&2{\rm Im}\int_{\mathbb{R}^{n}}\Delta P_{\varepsilon}P_{+}w_{n}\cdot \overline{P_{\varepsilon}P_{+}w_{n}}\ \d x,\\
			W_2:=&2{\rm Re}\int_{\mathbb{R}^{n}}P_{\varepsilon}P_{+}((\operatorname{div} \mathbb{A})w_{n})\cdot \overline{P_{\varepsilon}P_{+}w_{n}}\ \d x
			+2{\rm Im}\int_{\mathbb{R}^{n}}P_{\varepsilon}P_{+}(|\mathbb{A}|^{2}w_{n})\cdot \overline{P_{\varepsilon}P_{+}w_{n}}\ \d x\notag\\
			&+2{\rm Im}\int_{\mathbb{R}^{n}}P_{\varepsilon}P_{+}(\mathbb{V}w_{n})\cdot \overline{P_{\varepsilon}P_{+}w_{n}}\ \d x
			+2{\rm Im}\int_{\mathbb{R}^{n}}P_{\varepsilon}P_{+}(\tilde{\mathbb{F}}_{n})\cdot \overline{P_{\varepsilon}P_{+}w_{n}}\ \d x\notag\\
			&+2{\rm Im}\int_{\mathbb{R}^{n}}P_{\varepsilon}P_{+}(hw_{n})\cdot \overline{P_{\varepsilon}P_{+}w_{n}}\ \d x, \\ 
			W_3:=&2{\rm Im}\int_{\mathbb{R}^{n}}P_{\varepsilon}P_{+}(a^{2}\partial_{x_1}w_{n})\cdot \overline{P_{\varepsilon}P_{+}w_{n}}\ \d x
			+2{\rm Re}\int_{\mathbb{R}^{n}}P_{\varepsilon}P_{+}(tb\partial_{x_1}w_{n})\cdot \overline{P_{\varepsilon}P_{+}w_{n}}\ \d x,  \\  
			W_4:=&4{\rm Re}\int_{\mathbb{R}^{n}}P_{\varepsilon}P_{+}(\mathbb{A}\cdot \nabla w_{n})\cdot \overline{P_{\varepsilon}P_{+}w_{n}}\ \d x.   
		\end{align*}
		
		For the term $W_1$, integration by parts yields
		\begin{equation}\label{equ lemma1-12}
			{\rm Im}\int_{\mathbb{R}^{n}}\Delta P_{\varepsilon}P_{+}w_{n}\cdot \overline{P_{\varepsilon}P_{+}w_{n}}\ \d x=0.
		\end{equation}
		
		For the term $W_2$, since $w_{n}(\cdot)\in L^{2}(\mathbb{R}^{n})$ for almost every $t$, we use the H\"{o}lder inequality and the $L^{2}$ boundedness of $P_{\varepsilon}P_{+}$ to obtain the following estimates:
		\begin{align}
			\label{equ lemma1-13}\left| {\rm Re}\int_{\mathbb{R}^{n}}P_{\varepsilon}P_{+}((\operatorname{div} \mathbb{A})w_{n})\cdot \overline{P_{\varepsilon}P_{+}w_{n}}\ \d x\right|
			&\leqslant c\sum_{j=2}^{d}\|\partial_{x_{j}}\mathbb{A}_{j}\|_{L^{\infty}}\|w_n\|_{L^{2}}^{2}, \\
			\label{equ lemma1-14}\left|{\rm Im}\int_{\mathbb{R}^{n}}P_{\varepsilon}P_{+}(|\mathbb{A}|^{2}w_{n})\cdot \overline{P_{\varepsilon}P_{+}w_{n}}\ \d x\right|
			&\leqslant c\|\ |\mathbb{A}|^{2}\|_{L^{\infty}}\left\|w_{n}\right\|_{L^{2}}^{2}, \\
			\label{equ lemma1-15}\left|{\rm Im}\int_{\mathbb{R}^{n}}P_{\varepsilon}P_{+}(\mathbb{V}w_{n})\cdot \overline{P_{\varepsilon}P_{+}w_{n}}\ \d x\right|
			&\leqslant c\|\mathbb{V}\|_{L^{\infty}}\left\|w_{n}\right\|_{L^{2}}^{2}, \\
			\label{equ lemma1-16}\left|{\rm Im}\int_{\mathbb{R}^{n}}P_{\varepsilon}P_{+}(\tilde{\mathbb{F}}_{n})\cdot \overline{P_{\varepsilon}P_{+}w_{n}}\ \d x\right|
			&\leqslant c\|\tilde{\mathbb{F}}_{n}\|_{L^{2}}\left\|w_{n}\right\|_{L^{2}}, \\
			\label{equ lemma1-17}\left|{\rm Im}\int_{\mathbb{R}^{n}}P_{\varepsilon}P_{+}(hw_{n})\cdot \overline{P_{\varepsilon}P_{+}w_{n}}\ \d x\right|
			&\leqslant c\left\|h\right\|_{L^{\infty}}\left\|w_{n}\right\|_{L^{2}}^{2}\leqslant \frac{c}{n}\left\|w_{n}\right\|_{L^{2}}^{2},
		\end{align}
		where in \eqref{equ lemma1-16}, we used the fact that $\tilde{\mathbb{F}}_{n}(\cdot)\in L^{2}(\mathbb{R}^{n})$ for almost every $t$; and in \eqref{equ lemma1-17}, we used $\eqref{equ lemma1-5}$. The constant $c$ in $\eqref{equ lemma1-13}$-$\eqref{equ lemma1-17}$ is independent of $\varepsilon \in (0,1]$ and $n\in \mathbb{Z}^{+}$.
		
		Before estimating $W_3$ and $W_4$, we recall Calder${\rm \acute{o}}$n's first commutator estimates (see \cite{APC65,EMS93} or  \cite[Lemma 2.1]{KPV03}): Let $f\in L^2{}(\R^n), \partial_{x_{1}}a \in L^{\infty}(\R^n)$ and let  $P_{\pm}, P_{\varepsilon}$ be the operators defined in \eqref{projection}, then
		\begin{align}
			\label{equ lemma1-18}\left\|\left[P_{\pm};a\right]\partial_{x_{1}}f\right\|_{L^{2}},\,\,\left\|\partial_{x_{1}}\left[P_{\pm};a\right]f\right\|_{L^{2}}\leqslant& c\left\|\partial_{x_{1}}a\right\|_{L^{\infty}}\|f\|_{L^{2}}, \\
			\label{equ lemma1-20}\left\|\left[P_{\varepsilon};a\right]\partial_{x_{1}}f\right\|_{L^{2}},\,\,\left\|\partial_{x_{1}}\left[P_{\varepsilon};a\right]f\right\|_{L^{2}}\leqslant& c\left\|\partial_{x_{1}}a\right\|_{L^{\infty}}\|f\|_{L^{2}},
		\end{align}
		where the constant $c$ is independent of $\varepsilon \in (0,1]$ and $n\in \mathbb{Z}^{+}$.
		
		For the term $W_3$, we utilize the results established in \cite[Lemma 2.1]{KPV03}. We revisit these estimates and present them as the following statements:
		
		\noindent (\uppercase\expandafter{\romannumeral1}) \textit{For $a^{2}(x_1)$  satisfying \eqref{equ lemma1-6}}, we have
		\begin{align}\label{equ lemma1-22}
			{\rm Im}&\int_{\mathbb{R}^{n}}P_{\varepsilon}P_{+}(a^{2}\partial_{x_1}w_{n})\cdot \overline{P_{\varepsilon}P_{+}w_{n}}\ \d x \notag\\
			&={\rm Im}\int_{\mathbb{R}^{n}}\partial_{x_1}P_{\varepsilon}P_{+}(aw_{n})\cdot \overline{P_{\varepsilon}P_{+}(aw_{n})}\ \d x
			+O\bigg(\frac{\left\|w_{n}\right\|_{L^{2}}^{2}}{n}\bigg) \notag\\
			&= O\bigg(\frac{\left\|w_{n}\right\|_{L^{2}}^{2}}{n}\bigg)
		\end{align}
		\textit{holds uniformly in $\varepsilon \in (0,1]$ and $n\in \mathbb{Z}^{+}$}.
		
		\noindent  (\uppercase\expandafter{\romannumeral2}) \textit{For $b(x_1)$  satisfying \eqref{equ lemma1-7}}, we have
		\begin{align*}
			\int_{\mathbb{R}^{n}}&P_{\varepsilon}P_{+}(b\partial_{x_1}w_{n})\cdot \overline{P_{\varepsilon}P_{+}w_{n}}\ \d x \notag\\
			&=-\overline{\int_{\mathbb{R}^{n}}P_{\varepsilon}P_{+}(b\partial_{x_1}w_{n}\cdot \overline{P_{\varepsilon}P_{+}w_{n}})\ \d x}+O\bigg(\frac{\left\|w_{n}\right\|_{L^{2}}^{2}}{n}\bigg)
		\end{align*}
		\textit{\textit{holds uniformly in $\varepsilon \in (0,1]$ and $n\in \mathbb{Z}^{+}$}. Thus}
		\begin{equation}\label{equ lemma1-24}
			{\rm Re}\int_{\mathbb{R}^{n}}P_{\varepsilon}P_{+}(tb\partial_{x_1}w_{n})\cdot \overline{P_{\varepsilon}P_{+}w_{n}}\ \d x
			=O\bigg(\frac{\left\|w_{n}\right\|_{L^{2}}^{2}}{n}\bigg)
		\end{equation}
		\textit{holds uniformly in $\varepsilon \in (0,1]$ and $n\in \mathbb{Z}^{+}$}.
		
		\noindent From \eqref {equ lemma1-22} and \eqref{equ lemma1-24}, we derive the estimate for $W_3$. 
		
		For the term $W_4$, observe that
		\begin{equation}\label{W_4-sum}
			{\rm Re}(P_{\varepsilon}P_{+}(\mathbb{A}\cdot \nabla w_{n})\cdot \overline{P_{\varepsilon}P_{+}w_{n}})
			={\rm Re}\sum_{j=2}^{n}(P_{\varepsilon}P_{+}(\mathbb{A}_{j}\partial_{x_{j}}w_{n})\cdot \overline{P_{\varepsilon}P_{+}w_{n}}).
		\end{equation} 
		We address each term in the summation. First, from \eqref{equ lemma1-18}, we have
		\begin{equation*}
			\|P_{+}(\mathbb{A}_{j}\partial_{x_{j}}w_{n})-\mathbb{A}_{j} P_{+}(\partial_{x_{j}}w_{n})\|_{L^2}\leqslant c\|\partial_{x_{j}}\mathbb{A}_{j}\|_{L^{\infty}}\|w_n\|_{L^{2}}, \,\,\, j=2,\ldots,n.
		\end{equation*}
		This shows that the operator $P_{+}\mathbb{A}_{j}\partial_{x_{j}}-\mathbb{A}_{j} P_{+}\partial_{x_{j}}$,  acting on $w_n$, is bounded in $L^{2}$. Then by \eqref{equ lemma1-20}, we derive  
		\begin{equation*}
			\|P_{\varepsilon}P_{+}(\mathbb{A}_{j}\partial_{x_{j}}w_{n})-\mathbb{A}_{j} P_{\varepsilon}P_{+}(\partial_{x_{j}}w_{n})\|_{L^2}\leqslant c\|\partial_{x_{j}}\mathbb{A}_{j}\|_{L^{\infty}}\|w_n\|_{L^{2}},\,\,\, j=2,\ldots,n,
		\end{equation*}
		where the constant $c$ is independent of $\varepsilon\in (0,1]$ and $n\in \mathbb{Z}^{+}$.
		
		\noindent Hence, integrating by parts, and from \eqref{equ lemma1-18} and \eqref{equ lemma1-20}, we obtain
		\begin{align*}\label{equ lemma1-25}
			\int_{\mathbb{R}^{n}}P_{\varepsilon}P_{+}(\mathbb{A}_{j}\partial_{x_{j}}w_{n}) \overline{P_{\varepsilon}P_{+}w_{n}}\ \d x
			&=\int_{\mathbb{R}^{n}}\mathbb{A}_{j}P_{\varepsilon}P_{+}(\partial_{x_{j}}w_{n})\cdot \overline{P_{\varepsilon}P_{+}w_{n}}\ \d x
			+O\left(\|\partial_{x_{j}}\mathbb{A}_{j}\|_{L^{\infty}}\|w_n\|_{L^{2}}^{2}\right) \notag\\
			&=-\int_{\mathbb{R}^{n}}P_{\varepsilon}P_{+}w_{n}\overline{\partial_{x_{j}}(\mathbb{A}_{j}P_{\varepsilon}P_{+}w_{n})}\ \d x
			+O\left(\|\partial_{x_{j}}\mathbb{A}_{j}\|_{L^{\infty}}\|w_n\|_{L^{2}}^{2}\right) \notag\\
			&=-\int_{\mathbb{R}^{n}}P_{\varepsilon}P_{+}w_{n}\overline{P_{\varepsilon}P_{+}(\mathbb{A}_{j}\partial_{x_{j}}w_{n})}\ \d x
			+O\left(\|\partial_{x_{j}}\mathbb{A}_{j}\|_{L^{\infty}}\|w_n\|_{L^{2}}^{2}\right) \notag\\
			&=-\overline{\int_{\mathbb{R}^{n}}P_{\varepsilon}P_{+}(\mathbb{A}_{j}\partial_{x_{j}}w_{n}) \overline{P_{\varepsilon}P_{+}w_{n}}\ \d x}
			+O\left(\|\partial_{x_{j}}\mathbb{A}_{j}\|_{L^{\infty}}\|w_n\|_{L^{2}}^{2}\right).
		\end{align*}
		From the above inequality and \eqref{W_4-sum}, we deduce that
		\begin{equation}\label{equ lemma1-26}
			{\rm Re}\int_{\mathbb{R}^{n}}P_{\varepsilon}P_{+}(\mathbb{A}\cdot \nabla w_{n})\cdot \overline{P_{\varepsilon}P_{+}w_{n}}\ \d x
			\leqslant c\sum_{j=2}^{n}\|\partial_{x_{j}}\mathbb{A}_{j}\|_{L^{\infty}}\|w_n\|_{L^{2}}^{2},
		\end{equation}
		where the constant $c$ independent of $\varepsilon\in (0,1]$ and $n\in \mathbb{Z}^{+}$.
		
		Integrating \eqref{equ lemma1-11}, then from the estimates for $W_i$ ($1\leqslant i\leqslant 4$), i.e. $\eqref{equ lemma1-12}$-$\eqref{equ lemma1-17}$, $\eqref{equ lemma1-22}$-$\eqref{equ lemma1-26}$, we obtain the desired inequality $\eqref{equ lemma1-27}$. 
		
		Arguing similarly for $P_{-}$, we obtain
		\begin{align}\label{equ lemma1-28}
			\partial_{t}\int_{\mathbb{R}^{n}}|P_{\varepsilon}P_{-}w_{n}|^2\ \d x
			\geqslant& -6c\sum_{j=2}^{n}\|\partial_{x_{j}}\mathbb{A}_{j}\|_{L^{\infty}}\|w_n\|_{L^{2}}^{2}
			-2c\|\ |\mathbb{A}|^{2}\|_{L^{\infty}}\left\|w_{n}\right\|_{L^{2}}^{2} \notag\\
			\ \ &-2c\left\|\mathbb{V}\right\|_{L^{\infty}}\left\|w_{n}\right\|_{L^{2}}^{2}-2c\|\tilde{\mathbb{F}}\|_{L^{2}}\left\|w_{n}\right\|_{L^{2}}
			-\frac{c}{n}\left\|w_{n}\right\|_{L^{2}}^{2},
		\end{align}
		where $c$ is independent of $\varepsilon \in (0,1]$ and $n\in \mathbb{Z}^{+}$.

		\textbf{Step 3: Estimating the $L^{2}$ norm of $w_n$.}
		
		We observe that for each $n\in \mathbb{Z}^{+}$, 
		\begin{equation}\label{equ lemma1-29}
			\sup_{0\leqslant t\leqslant 1}\|w_n(\cdot,t)\|_{L^{2}}<\infty,
		\end{equation}
		which implies the existence of $t_{n}\in [0,1]$ such that
		\begin{equation}\label{equ lemma1-30}
			\frac{1}{2}\sup_{0\leqslant t\leqslant 1}\|w_n(\cdot,t)\|_{L^{2}}^{2}\leqslant \|w_n(\cdot,t_n)\|_{L^{2}}^{2}.
		\end{equation}
		Thus, from $\eqref{equ lemma1-27}$, $\eqref{equ lemma1-28}$-$\eqref{equ lemma1-30}$, we obtain
		\begin{align}\label{equ lemma1-31}
			\frac{1}{2}&\sup_{0\leqslant t\leqslant 1}\|w_n(\cdot,t)\|_{L^{2}}^{2}\leqslant \|w_n(\cdot,t_n)\|_{L^{2}}^{2}\notag\\
			&=\lim_{\varepsilon\rightarrow 0}\left(\|P_{\varepsilon}P_{+}w_n(\cdot,t_n)\|_{L^{2}}^{2}+\|P_{\varepsilon}P_{-}w_n(\cdot,t_n)\|_{L^{2}}^{2}\right)\notag\\
			&=\lim_{\varepsilon\rightarrow 0}\left(\int_{0}^{t_n}\partial_{s}\|P_{\varepsilon}P_{+}w_n(\cdot,s)\|_{L^{2}}^{2}\ \d s
			+\|P_{\varepsilon}P_{+}w_n(\cdot,0)\|_{L^{2}}^{2}\right.\notag\\
			&\left.\ \ \ \ \ \ \ \ \ \ \ \ \ -\int_{t_n}^{1}\partial_{s}\|P_{\varepsilon}P_{-}w_n(\cdot,s)\|_{L^{2}}^{2}\ \d s+\|P_{\varepsilon}P_{-}w_n(\cdot,1)\|_{L^{2}}^{2}\right)\notag\\
			&\leqslant 6c\sum_{j=2}^{n}\int_{0}^{1}\|\partial_{x_{j}}\mathbb{A}_{j}\|_{L^{\infty}}\ \d s\cdot \sup_{0\leqslant t\leqslant 1}\|w_n(\cdot,t)\|_{L^{2}}^{2}
			+2c\int_{0}^{1}\|\ |\mathbb{A}|^{2}\|_{L^{\infty}}\ \d s\cdot \sup_{0\leqslant t\leqslant 1}\left\|w_{n}(\cdot,t)\right\|_{L^{2}}^{2}\notag\\
			&\ \ \ \ +2c\int_{0}^{1}\left\|\mathbb{V}\right\|_{L^{\infty}}\ \d s\cdot \sup_{0\leqslant t\leqslant 1}\|w_n(\cdot,t)\|_{L^{2}}^{2}
			+2c\int_{0}^{1}\left\|\mathbb{F}\right\|_{L^{2}(e^{2\beta x_1}dx)}\ \d s\cdot \sup_{0\leqslant t\leqslant 1}\|w_n(\cdot,t)\|_{L^{2}}\notag\\
			&\ \ \ \ +\frac{c}{n}\sup_{0\leqslant t\leqslant 1}\|w_n(\cdot,t)\|_{L^{2}}^{2}+\|w_n(\cdot,0)\|_{L^{2}}^{2}+\|w_n(\cdot,1)\|_{L^{2}}^{2}.
		\end{align}
		\noindent Taking $n$ large enough such that $\frac{c}{n}<\frac{1}{16}$ and choosing $\varepsilon_0$, $\varepsilon'_{0}$ and $\varepsilon''_{0}$ such that
		\begin{equation*}
			6c\sum_{j=2}^{n}\int_{0}^{1}\|\partial_{x_{j}}\mathbb{A}_{j}\|_{L^{\infty}}\ \d s<\frac{1}{16},\,\qquad 2c\int_{0}^{1}\|\ |\mathbb{A}|^{2}\|_{L^{\infty}}\ \d s<\frac{1}{16},
		\end{equation*}
		and 
		\begin{equation*}
			2c\int_{0}^{1}\left\|\mathbb{V}\right\|_{L^{\infty}}\ \d s<\frac{1}{16}.   
		\end{equation*}
		Under these conditions, \eqref{equ lemma1-31} is transformed into
		\begin{equation*}
			\frac{1}{4}\sup_{0\leqslant t\leqslant 1}\|w_n(\cdot,t)\|_{L^{2}}^{2}\leqslant 16c^{2}\left(\int_{0}^{1}\left\|\mathbb{F}\right\|_{L^{2}(e^{2\beta x_1}dx)}\ \d s\right)^{2}+\|w_n(\cdot,0)\|_{L^{2}}^{2}+\|w_n(\cdot,1)\|_{L^{2}}^{2}.
		\end{equation*}
		Taking the limit as $n\rightarrow\infty$, we obtain the desired inequality \eqref{equ lemma1-1}. This completes the proof  of Lemma \ref{Interpolation}.
	\end{proof}
	
	The next lemma gives an appropriate Carleman estimate for the magnetic Schr\"{o}dinger operator $\i\partial_{t}+\Delta_{A}$.
	
	\begin{lemma}\label{carleman estimate}
		Let $n\geqslant 3$, and let $\varphi:[0,1]\rightarrow \mathbb{R}$ be a smooth function. Define $A=A_k=A_k(x,t):\mathbb{R}^{n+1}\rightarrow \mathbb{R}^{n}$ and denote $B=B_k=B_k(x,t)=DA_k-DA_k^{t}$, where $k$ is a parameter. Assume that there exists some large and fixed constant $k_0$, such that
		\begin{equation}\label{RXB}
			\|x^{t}B\|_{L^{\infty}}\leqslant \Big(\frac{k^m}{a_0}\Big)^{\frac{1}{2p}}M_B, \,\,\,\,\,\mbox{if}\,\,\, k\geqslant k_0,
		\end{equation}
		where $a_0,m$ are the constants in \eqref{thm1.1-6}-\eqref{thm1.1-7}, $M_B$ is defined in \eqref{M_B}. Additionally,  assume that
		\begin{equation}\label{carleman1-1}
			x\cdot \partial_{t}A(x,t)=0,\ \  e_{1}\cdot \partial_{t}A(x,t)=0, \ \  e_{1}^{t}B(x,t)=0,
		\end{equation}
		for any $x\in \mathbb{R}^{n}$ and $e_{1}=(1,0,\ldots,0)$. Then there exists $c=c(n, M_B, \|\varphi'\|_{L^{\infty}}, \|\varphi''\|_{L^{\infty}}), c_1=c_1(n)$ such that the inequality 
		\begin{equation}\label{carleman1-2}
			\frac{\sigma^{\frac{3}{2}}}{R^2}\|e^{\sigma|\frac{x}{R}+\varphi(t)e_{1}|^{2}}g\|_{L^{2}(\d x \ \d t)}\leqslant c_1 \|e^{\sigma|\frac{x}{R}+\varphi(t)e_{1}|^{2}}(\i\partial_t+\Delta_{A})g\|_{L^{2}(\d x \ \d t)}
		\end{equation}
		holds when $\sigma \geqslant cR^{2}$, where
		\begin{equation}\label{RKA0}
			R=2a_0^{-\frac{1}{2p}}k^{\frac{m}{2(2-p)}},    
		\end{equation}
		with $k,a_0,m$ being the parameters mentioned above, and $g\in C_{0}^{\infty}(\mathbb{R}^{n+1})$ with 
		\begin{equation}\label{supp g}
			supp \  g \subset \left\{(x,t):\big|\frac{x}{R}+\varphi(t)e_1\big| \geqslant 1 \right\}.  
		\end{equation}
	\end{lemma}
	
	\begin{proof}
		We point out that the analogous result has been proved in \cite[Lemma 2.3]{BCF22}, with the exception that, in our context, the potential $A_k$ and the magnetic field $B_k$ also depend on the parameter $k$. For the sake of completeness, we give the proof here.
		Set $\phi(x,t)=\sigma|\frac{x}{R}+\varphi(t)e_{1}|^{2}$ and $f(x,t):=e^{\phi}g(x,t)$, then
		\begin{equation}\label{carleman1-3}
			e^{\phi}(\i\partial_t+\Delta_{A})g=\mathcal{S}_{\sigma}f+\mathcal{A}_{\sigma}f,
		\end{equation}
		where 
		\begin{align*}
			\mathcal{S}_{\sigma}&=\i\partial_{t}+\Delta_{A}+|\nabla_x \phi|^{2}, \\
			\mathcal{A}_{\sigma}&=-2\nabla_x \phi\cdot\nabla_{A}-\Delta_x \phi-\i\partial_{t}\phi.
		\end{align*}
		Thus, 
		\begin{equation}\label{carleman1-4}
			\mathcal{S}_{\sigma}^{*}=\mathcal{S}_{\sigma},\ \ \ \mathcal{A}_{\sigma}^{*}=-\mathcal{A}_{\sigma},
		\end{equation}
		and
		\begin{align}\label{carleman1-5}
			\|e^{\phi}(\i\partial_t+\Delta_{A})g\|_{L^{2}}^{2}&=\langle\mathcal{S}_{\sigma}f+\mathcal{A}_{\sigma}f, \mathcal{S}_{\sigma}f+\mathcal{A}_{\sigma}f\rangle\notag\\
			&\geqslant \langle(\mathcal{S}_{\sigma}\mathcal{A}_{\sigma}-\mathcal{A}_{\sigma}\mathcal{S}_{\sigma})f,f\rangle=\langle[\mathcal{S}_{\sigma},\mathcal{A}_{\sigma}]f,f\rangle.
		\end{align}
		Note that
		\begin{align*}
			\partial_{t}\phi=&2\sigma(\frac{x_1}{R}+\varphi e_{1})\varphi',\  \partial_{tt}\phi=2\sigma(\frac{x_1}{R}+\varphi)\varphi''+2\sigma \varphi'^{2},\\
			\nabla_{x}\phi=&\frac{2\sigma}{R}(\frac{x}{R}+\varphi e_{1}),\  \nabla_{x}\partial_t\phi=(\frac{2\sigma}{R}\varphi',0,\ldots,0), \\
			\Delta_{x}\phi=&\frac{2d\sigma}{R^{2}},\  \Delta_{x}^{2}\phi=0, \ \Delta_{x}^{2}\partial_{t}\phi=0, \ D^{2}\phi=\frac{2d\sigma}{R^{2}}I,
		\end{align*}
		a calculation shows that
		\begin{align}\label{carleman1-6}
			\langle&[\mathcal{S}_{\sigma},\mathcal{A}_{\sigma}]f,f\rangle \notag\\
			&=\frac{32\sigma^3}{R^4}\int \big|\frac{x}{R}+\varphi e_{1}\big|^{2}|f|^{2}\ \d x\ \d t+\frac{8\sigma}{R^2}\int |\nabla_{A}f|^{2}\ \d x\ \d t \notag\\
			&\ \ \ +2\sigma\int \Big[(\frac{x_1}{R}+\varphi)\varphi''+\varphi'^2\Big]|f|^{2}\ \d x\ \d t
			+\frac{8\sigma}{R}{\rm Im}\int \varphi'(\nabla_{A}\cdot e_1)f\overline{f}\ \d x\ \d t\notag\\
			&\ \ \ -\frac{4\sigma}{R}\int (\frac{x}{R}+\varphi e_{1})\cdot \partial_{t}A |f|^{2}\ \d x\ \d t
			-\frac{8\sigma}{R}{\rm Im}\int f(\frac{x}{R}+\varphi e_{1})^{t}B\cdot \overline{\nabla_{A}f}\ \d x\ \d t.
		\end{align}
		By \eqref{carleman1-1}, the last but one term at right hand side of \eqref{carleman1-6} vanishes. Also, we have
		\begin{align}\label{carleman1-7}
			\frac{8\sigma}{R}{\rm Im}\int \varphi'(\nabla_{A}\cdot e_1)f\overline{f}\ \d x\ \d t
			\geqslant -4\sigma \|\varphi'\|_{L^{\infty}}^{2}\int |f|^{2}\ \d x \ \d t-\frac{4\sigma}{R^2}\int|\nabla_{A}f|^2 \ \d x\ \d t.
		\end{align}
		In addition, by \eqref{carleman1-1}, we have
		\begin{align}\label{carleman1-8}
			-&\frac{8\sigma}{R}{\rm Im}\int f(\frac{x}{R}+\varphi e_{1})^{t}B\cdot \overline{\nabla_{A}f}\ \d x\ \d t
			=-\frac{8\sigma}{R^2}{\rm Im}\int fx^{t}B\cdot \overline{\nabla_{A}f}\ \d x\ \d t\notag\\
			&\geqslant -\frac{4\sigma}{R^2}\|x^{t}B\|_{L^{\infty}}^{2}\int |f|^{2}\ \d x \ \d t-\frac{4\sigma}{R^2}\int|\nabla_{A}f|^2 \ \d x\ \d t.
		\end{align}
		Using the hypothesis on the support of $g$, i.e. \eqref{supp g}, we have
		\begin{align}\label{carleman1-9}
			&\|e^{\phi}(\i\partial_t+\Delta_{A})g\|_{L^{2}}^{2}\notag\\
			&\ \ \geqslant \Big[\frac{32\sigma^3}{R^4}-2\sigma\big(\|\varphi''\|_{L^{\infty}}+\|\varphi'\|_{L^{\infty}}^2+\frac{2}{R^2}k^{\frac{m}{p}}a_0^{-\frac{1}{p}}M_B^2\big)\Big]
			\int \big|\frac{x}{R}+\varphi e_{1}\big||f|^{2}\d x\ \d t\notag\\
			&\ \ \geqslant \Big[\frac{32\sigma^3}{R^4}-2\sigma\big(\|\varphi''\|_{L^{\infty}}+\|\varphi'\|_{L^{\infty}}^2+\frac{1}{2}M_B^2\big)\Big]
			\int \big|\frac{x}{R}+\varphi e_{1}\big||f|^{2}\d x\ \d t,
		\end{align}
		where in the first inequality, we used \eqref{RXB}, \eqref{carleman1-5}-\eqref{carleman1-8}; while in the second inequality, we used \eqref{RKA0}. 
		
		\eqref{carleman1-2} follows from \eqref{carleman1-9} if  $\sigma\geqslant cR^2$ for some large $c=c(n, M_B, \|\varphi'\|_{L^{\infty}}, \|\varphi''\|_{L^{\infty}})$. Therefore the proof of  Lemma \ref{carleman estimate} is complete. 
	\end{proof}

	\section{Proofs of the main results}\label{proof of corollary and theorem}
	In Subsection \ref{proof-thm3.1}, we first establish  a modified version  of Theorem \ref{morgan for p-q}. Then, as a corollary, we  proceed to prove Theorem \ref{morgan for p-q} and Corollary \ref{apply morgan} in Subsections \ref{sec-4-2} and \ref{sec-4-3}, respectively.
	
	\subsection{A variant form of Theorem \ref{morgan for p-q}.}\label{proof-thm3.1}\,
	
	\begin{theorem}\label{morgan for p-p}
		Let $n\geqslant 3$ and $1<p<2$. There exists some $M_{p}>0$ such that for any solution  $u\in C([0,1];L^{2}(\mathbb{R}^n))$  of \eqref{Introduction-1} that satisfies \textbf{Condition A} and meets the following criteria for some positive constants $a_0, a_1, a_2>0$:
		\begin{equation}\label{thm1.1-6}
			\int_{\mathbb{R}^{n}} |u(x,0)|^{2}e^{2a_0|x|^p}\ \d x<\infty,
		\end{equation}
		and for any fixed $m\in \R^{+}$,
		\begin{equation}\label{thm1.1-7}
			\int_{\mathbb{R}^{n}} |u(x,1)|^{2}e^{2k^{m}|x|^p}\ \d x<a_2^2e^{2a_1k^{mq/(q-p)}},\,\,\,\,\,\, \mbox{for any} \,\,\, k\in \mathbb{Z^{+}},
		\end{equation}
		where  $1/p+1/q=1$. Additionally, if 
		\begin{equation}\label{thm1.1-8}
			a_0a_1^{p-2}>M_p,
		\end{equation}
		then $u\equiv 0$.
	\end{theorem}
	\begin{proof}
		First, due to  gauge invariance, it is sufficient to prove Theorem \ref{morgan for p-p} for the function $\tilde{u}=e^{\i\varphi}u$, where $\tilde{u}$ is a solution to
		\begin{equation}\label{pf-of-thm1.1-1}
			\partial_{t}\tilde{u}=\i(\Delta_{\tilde{A}}\tilde{u}+V(x,t)\tilde{u}+\tilde{F}(x,t)),
		\end{equation}
		with $\tilde{F}=e^{\i\varphi}F$ and $\varphi$, $\tilde{A}$ defined as in \eqref{gauge1-2}, \eqref{gauge1-3}, respectively.
		By Lemma \ref{Cronstrom gauge}, it follows that
		\begin{equation}\label{pf-of-thm1.1-3/2}
			x\cdot \tilde{A}(x)=0,\,\,\, x\cdot x^{t}D\tilde{A}(x)=0.
		\end{equation}
		Thus we  reduce to the case of the Cronstr\"{o}m gauge.
		Moreover, from \eqref{thm1.1-4} and \eqref{gauge1-3}, we see that
		\begin{equation}\label{pf-of-thm1.1-2}
			e_{1}\cdot \tilde{A}(x)=0, \,\,\,\,\,\mbox{for all} \,\,\, x\in \mathbb{R}^{n}.
		\end{equation}
		
		Next, we  apply Lemma \ref{Appell trans}  (the Appell transformation) to solutions of \eqref{pf-of-thm1.1-1}. To simplify the notation, we will henceforth omit the tildes and denote $\tilde{u}, \tilde{A}$, $\tilde{F}$ by $u, A$, $F$, respectively.
		Using  the Appell transform \eqref{Appell1-2}, it follows from \eqref{Appell1-3} and  $x\cdot A(x)=0$ (due to \eqref{pf-of-thm1.1-3/2}) that $\tilde{u}$ solves
		\begin{equation}\label{pf-of-thm1.1-3}
			\partial_{t}\tilde{u}=\i\bigg(\Delta_{\tilde{A}}\tilde{u}+\tilde{V}(x,t)\tilde{u}+\tilde{F}(x,t)\bigg) \,\,\,\,\mbox{in} \,\,\,\mathbb{R}^{n}\times [0,1],
		\end{equation}
		where
		\begin{align}
			\label{pf-of-thm1.1-4}\tilde{A}(x,t)&=\frac{\sqrt{\alpha\beta}}{\alpha(1-t)+\beta t}A\left(\frac{\sqrt{\alpha \beta}x}{\alpha(1-t)+\beta t}\right),\\
			\label{pf-of-thm1.1-5}\tilde{V}(x,t)&=\frac{\alpha\beta}{(\alpha(1-t)+\beta t)^{2}}V\left(\frac{\sqrt{\alpha \beta}x}{\alpha(1-t)+\beta t}, \frac{\beta t}{\alpha(1-t)+\beta t}\right),\\
			\label{pf-of-thm1.1-6}\tilde{F}(x,t)&=\left(\frac{\sqrt{\alpha\beta}}{\alpha(1-t)+\beta t}\right)^{\frac{n}{2}+2}F\left(\frac{\sqrt{\alpha \beta}x}{\alpha(1-t)+\beta t}, \frac{\beta t}{\alpha(1-t)+\beta t}\right)e^{\frac{(\alpha-\beta)|x|^{2}}{4\i(\alpha(1-t)+\beta t)}}.
		\end{align}
		Observe that $\tilde{A}$ is also in the Cronstr\"{o}m gauge. In addition, by \eqref{pf-of-thm1.1-3/2}, \eqref{pf-of-thm1.1-2} and \eqref{pf-of-thm1.1-4}, we derive that
		\begin{equation}\label{pf-of-thm1.1-7}
			e_{1}\cdot \tilde{A}(x,t)=0, \ \ e_{1}\cdot \partial_{t}\tilde{A}(x,t)=0, \ \ x\cdot \partial_{t}\tilde{A}(x,t)=0, \,\,\, \mbox{ for all} \,\, x\in \mathbb{R}^{n},\, t\in [0,1].
		\end{equation}
		
		Now, we divide the proof into three steps.
		
		\textbf{Step 1: The upper bounds.}
		
		We prove that
		\begin{align}\label{pf-of-thm1.1-33}
			\sup_{[0,1]}\|e^{\gamma |x|^{p}}\tilde{u}(t)\|_{L^{2}}^{2}+\gamma\int_{0}^{1}\int_{\mathbb{R}^{n}}\frac{t(1-t)}{(1+|x|)^{2-p}}|\nabla_{\tilde{A}}\tilde{u}(x,t)|^{2}e^{\gamma |x|^{p}}\ \d x \ \d t
			\leqslant c^{*}k^{c_{p,m}}e^{2a_{1}k^{\frac{m}{2-p}}}
		\end{align}
		holds for $k\geqslant k_{0}$, where $k, k_0,m, c^*,c_{p,m}$ as in \eqref{MBMV}-\eqref{bounded1-4}, $\gamma=(k^m a_0)^{\frac{1}{2}}$, and $a_0, a_1$ are the constants in \eqref{thm1.1-6}-\eqref{thm1.1-7}.
		
		First, we choose $\beta=\beta(k)$. By hypothesis on $u(0)$ and $u(1)$, we have
		\begin{align}
			\label{pf-of-thm1.1-8}\|e^{a_0|y|^p}u(y,0)\|_{L^{2}}&=: A_0,\\
			\label{pf-of-thm1.1-9}\|e^{k^m|y|^p}u(y,1)\|_{L^{2}}&=: A_k\leqslant a_2e^{a_1k^{\frac{mq}{q-p}}}=a_2e^{a_1k^{\frac{m}{2-p}}}.
		\end{align}
		Thus, for $\gamma=\gamma(k)\in [0,\infty)$ to be chosen later, we obtain from  \eqref{Appell1-2} that
		\begin{align}
			\label{pf-of-thm1.1-10}\|e^{\gamma|x|^p}\tilde{u}(x,0)\|_{L^{2}}&=\|e^{\gamma(\frac{\alpha}{\beta})^{\frac{p}{2}}|x|^p}u(x,0)\|_{L^{2}}=A_0,\\
			\label{pf-of-thm1.1-11}\|e^{\gamma|x|^p}\tilde{u}(x,1)\|_{L^{2}}&=\|e^{\gamma(\frac{\beta}{\alpha})^{\frac{p}{2}}|x|^p}u(x,1)\|_{L^{2}}=A_k.
		\end{align}
		To match our assumptions, we now take
		\begin{equation*}
			\gamma\left(\frac{\alpha}{\beta}\right)^{\frac{p}{2}}=a_0 \ \ {\rm and} \ \ \gamma\left(\frac{\beta}{\alpha}\right)^{\frac{p}{2}}=k^m.
		\end{equation*}
		Therefore, we have
		\begin{equation}\label{pf-of-thm1.1-12}
			\alpha=a_0^{\frac{1}{p}},\ \ \beta=k^{\frac{m}{p}}, \ \ \gamma=(k^ma_0)^{\frac{1}{2}}.
		\end{equation}
		From \eqref{pf-of-thm1.1-3}, we get
		\begin{equation}\label{pf-of-thm1.1-13}
			M:=\int_{0}^{1}\|{\rm Im}\ V(t)\|_{L^{\infty}}\ \d t=\int_{0}^{1}\|{\rm Im}\ \tilde{V}(s)\|_{L^{\infty}}\ \d s.
		\end{equation}
		Using energy estimates and taking $F=0$, we obtain
		\begin{equation}\label{pf-of-thm1.1-14}
			e^{-M}\|u(0)\|_{L^{2}}\leqslant\|u(t)\|_{L^{2}}= \|\tilde{u}(s)\|_{L^{2}}\leqslant e^{M}\|u(0)\|_{L^{2}},
		\end{equation}
		where $t, s\in [0,1]$ and $s=\frac{\beta t}{\alpha(1-t)+\beta t}$.
		
		Next, we desire to apply Lemma \ref{Interpolation} to a solution of \eqref{pf-of-thm1.1-3}. Since $0<\alpha<\beta=\beta(k)$ for $k\geqslant k_{0}$, then 
		\begin{equation*}
			\alpha\leqslant \alpha(1-t)+\beta t\leqslant \beta, \ \ \ \mbox{for any} \,\, t\in [0,1].
		\end{equation*}
		Hence, if $y=\frac{\sqrt{\alpha\beta}x}{\alpha(1-t)+\beta t}$, then by \eqref{pf-of-thm1.1-12}, we have
		\begin{equation}\label{pf-of-thm1.1-15}
			\bigg(\frac{a_0}{k^m}\bigg)^{\frac{1}{2p}}|x|=\sqrt{\frac{\alpha}{\beta}}|x|\leqslant |y|\leqslant\sqrt{\frac{\beta}{\alpha}}|x|=\bigg(\frac{k^m}{a_0}\bigg)^{\frac{1}{2p}}|x|.
		\end{equation}
		Thus, for $j=2,\ldots,n$,
		\begin{align}\label{pf-of-thm1.1-16}
			\|\tilde{V}\|_{L^{\infty}}\!\leqslant\! \bigg(\frac{k^m}{a_0}\bigg)^{\frac{1}{p}}\!\|V\|_{L^{\infty}},\,
			\|\ |\tilde{A}|^{2}\|_{L^{\infty}}\!
			\leqslant\! \bigg(\frac{k^m}{a_0}\bigg)^{\frac{1}{p}}\!
			\| \ |A|^{2}\|_{L^{\infty}},\,
			\|\partial_{x_j}\tilde{A}_j\|_{L^{\infty}}\!\leqslant\! \bigg(\frac{k^m}{a_0}\bigg)^{\frac{1}{p}}\!\|\partial_{x_j}A_j\|_{L^{\infty}}.
		\end{align}
		Also, if $s=\frac{\beta t}{\alpha(1-t)+\beta t}$, then $\d t=\frac{(\alpha(1-t)+\beta t)^{2}}{\alpha\beta}\d s$. Consequently,
		\begin{align}\label{pf-of-thm1.1-17}
			\int_{0}^{1}\|\tilde{V}(\cdot,t)\|_{L^{\infty}}\ \d t&=\int_{0}^{1}\|V(\cdot,s)\|_{L^{\infty}}\ \d s,\notag\\
			\int_{0}^{1}\|\ |\tilde{A}|^{2}(\cdot,t)\|_{L^{\infty}}\ \d t&=\|\ |A|^{2}(\cdot)\|_{L^{\infty}},\notag\\
			\int_{0}^{1}\|\partial_{x_j}\tilde{A}_j(\cdot,t)\|_{L^{\infty}}\ \d t&=\|\partial_{x_j}A_j(\cdot)\|_{L^{\infty}},
		\end{align}
		and by \eqref{pf-of-thm1.1-15}, we have
		\begin{equation}\label{pf-of-thm1.1-18}
			\int_{0}^{1}\|\tilde{V}(\cdot,t)\|_{L^{\infty}(|x|\geqslant R)}\ \d t\leqslant \int_{0}^{1}\|V(\cdot,s)\|_{L^{\infty}(|y|>\Gamma)}\ \d s,
		\end{equation}
		where $\Gamma=\Big(\frac{a_0}{k^m}\Big)^{\frac{1}{2p}}R$. Hence, if
		\begin{equation*}
			\int_{0}^{1}\|V(\cdot,s)\|_{L^{\infty}(|y|>\Gamma)}\ \d s\leqslant \varepsilon_{0},
		\end{equation*}
		then
		\begin{equation}\label{pf-of-thm1.1-19}
			\int_{0}^{1}\|\tilde{V}(\cdot,t)\|_{L^{\infty}(|x|\geqslant R)}\ \d t\leqslant \varepsilon_{0}, \ \ \ R=\Gamma\Big(\frac{k^m}{a_0}\Big)^{\frac{1}{2p}}.
		\end{equation}
		We apply Lemma \ref{Interpolation} to \eqref{pf-of-thm1.1-3} with
		\begin{equation}\label{pf-of-thm1.1-20}
			\mathbb{A}(x,t)=\tilde{A}(x,t),\ \mathbb{V}(x,t)=\tilde{V}(x,t)\chi_{(|x|>R)}(x),\ \mathbb{F}(x,t)=\tilde{V}(x,t)\chi_{(|x|\leqslant R)}(x)\tilde{u}(x,t),
		\end{equation}
		and recall the following fact (see e.g. in \cite{EKPV11}): for any $x\in \mathbb{R}^{n}$, $p\in (1,2)$, $\frac{1}{p}+\frac{1}{q}=1$, 
		\begin{equation}\label{pf-of-thm1.1-21}
			e^{\gamma |x|^{p}/p} \simeq \int_{\mathbb{R}^{n}}e^{\gamma^{1/p}\lambda\cdot x-|\lambda|^{q}/q}|\lambda|^{n(q-2)/2}\ \d \lambda.
		\end{equation}
		Substituting $(2p)^{\frac{1}{p}}\gamma^{\frac{1}{p}}\lambda/2$ for $\lambda$ in \eqref{equ lemma1-0}, we obtain
		\begin{align}\label{pf-of-thm1.1-22}
			\sup_{[0,1]}\|e^{(2p)^{\frac{1}{p}}\gamma^{\frac{1}{p}}\lambda\cdot x/2}\tilde{u}(t)\|_{L^{2}}
			\leqslant& c_{n}\Big(\|e^{(2p)^{\frac{1}{p}}\gamma^{\frac{1}{p}}\lambda\cdot x/2}\tilde{u}(0)\|_{L^{2}}
			+\|e^{(2p)^{\frac{1}{p}}\gamma^{\frac{1}{p}}\lambda\cdot x/2}\tilde{u}(1)\|_{L^{2}}\Big)\notag\\
			&+c_{n}\|\tilde{V}\|_{L^{\infty}}\|u(0)\|_{L^{2}}e^{M}e^{|\lambda|(2p)^{\frac{1}{p}}\gamma^{\frac{1}{p}}R/2}.
		\end{align}
		Now we square both sides of \eqref{pf-of-thm1.1-22}, multiply by $e^{-|\lambda|^{q}/q}|\lambda|^{n(q-2)/2}$, and integrate over $\lambda$ and  $x$. Applying Fubini theorem and \eqref{pf-of-thm1.1-21}, we  obtain
		\begin{align}\label{pf-of-thm1.1-23}
			\int_{|x|>1}e^{2\gamma |x|^{p}}|\tilde{u}(x,t)|^{2}\ \d x
			\leqslant& c_{n}\int_{\mathbb{R}^{n}}e^{2\gamma |x|^{p}}\left(|\tilde{u}(x,0)|^{2}+\tilde{u}(x,1)|^{2}\right)\ \d x\notag\\
			&+c_{n}\|u(0)\|_{L^{2}}^{2}\|\tilde{V}\|_{L^{\infty}}^{2}e^{2M}e^{2\gamma R^{p}}.
		\end{align}
		Then we have
		\begin{align}\label{pf-of-thm1.1-24}
			\sup_{[0,1]}\|e^{\gamma |x|^{p}}\tilde{u}(t)\|_{L^{2}}\leqslant&c_{n}\left(\|e^{\gamma |x|^{p}}\tilde{u}(0)\|_{L^{2}}+\|e^{\gamma |x|^{p}}\tilde{u}(1)\|_{L^{2}}\right)\notag\\
			&+c_{n}\|u(0)\|_{L^{2}}e^{M}e^{\gamma}+c_{n}\|u(0)\|_{L^{2}}\Big(\frac{k^m}{a_0}\Big)^{c_p}\|V\|_{L^{\infty}}e^{M}e^{\Gamma^{p}\gamma(\frac{k^m}{a_0})^{\frac{1}{2}}}\notag\\
			\leqslant& c_{n}(A_0+A_k)+c_{n}\|u(0)\|_{L^{2}}e^{M}\bigg(e^{\gamma}+\Big(\frac{k^m}{a_0}\Big)^{c_p}\|V\|_{L^{\infty}}e^{\Gamma^{p}k^m}\bigg)\notag\\
			\leqslant& c^{*}k^{c_{p,m}}e^{a_{1}k^{\frac{m}{2-p}}},\,\,\,\, \mbox{if} \,\,\, k\geqslant k_{0},
		\end{align}
		where in the first inequality, we used \eqref{pf-of-thm1.1-14}, \eqref{pf-of-thm1.1-16} and \eqref{pf-of-thm1.1-23}; in the second inequality, we used \eqref{pf-of-thm1.1-12}; while in the last inequality, we used \eqref{pf-of-thm1.1-8}, \eqref{pf-of-thm1.1-9}. The constants $k, k_0,m, c^*,c_{p,m}$ are given by \eqref{MBMV}-\eqref{bounded1-4}.
		
		It remains to establish bounds for the term $\nabla_{\tilde{A}}\tilde{u}$. We shall use
		\begin{equation}\label{pf-of-thm1.1-25}
			\partial_{t}\tilde{u}=\i\Delta_{\tilde{A}}\tilde{u}+\i F,\ \ \ F=\tilde{V}\tilde{u},
		\end{equation}
		and let
		\begin{equation}\label{pf-of-thm1.1-26}
			f(x,t)=e^{\tilde{\gamma}\varphi}\tilde{u}(x,t).
		\end{equation}
		Substituting $\tilde{A}$, $\tilde{V}$, $\tilde{u}$, $\gamma$, $f$ for $A$, $V$, $v$, $\theta$, $h$ in Proposition \ref{bound h}, we obtain
		\begin{align*}\label{pf-of-thm1.1-27}
			8\tilde{\gamma}&\int_{0}^{1}\int_{\mathbb{R}^{n}}t(1-t)\nabla_{\tilde{A}}f\cdot D^{2}\varphi \overline{\nabla_{\tilde{A}}f}\ \d x\ \d t+8{\tilde{\gamma}}^{3}\int_{0}^{1}\int_{\mathbb{R}^{n}}t(1-t)D^{2}\varphi \nabla \varphi\cdot \nabla \varphi|f|^{2}\ \d x \ \d t\notag\\
			&\leqslant c^{*}k^{c_{p,m}}e^{2a_{1}k^{\frac{m}{2-p}}},\,\,\,\, \mbox{if} \,\,\, k\geqslant k_{0},
		\end{align*}
		where $k, k_0,m, c^*,c_{p,m}$ are given by \eqref{MBMV}-\eqref{bounded1-4}. 
		By \eqref{pf-of-thm1.1-26}, we have
		\begin{equation*}\label{pf-of-thm1.1-28}
			\nabla_{\tilde{A}}f=\tilde{\gamma}\nabla \varphi e^{\tilde{\gamma}\varphi}\tilde{u}+e^{\tilde{\gamma}\varphi}\nabla_{\tilde{A}}{\tilde{u}}.
		\end{equation*}
		In addition,
		\begin{equation*}\label{pf-of-thm1.1-29}
			|2\tilde{\gamma}^{2}D^{2}\varphi\nabla \varphi\cdot \nabla_{\tilde{A}}{\tilde{u}}{\tilde{u}}|
			\leqslant\frac{3}{2}\tilde{\gamma}^{3}D^{2}\varphi\nabla \varphi\cdot \nabla\varphi|{\tilde{u}}|^{2}+\frac{2}{3}\tilde{\gamma}D^{2}\varphi \nabla_{\tilde{A}}{\tilde{u}}\cdot \nabla_{\tilde{A}}{\tilde{u}}.
		\end{equation*}
		Therefore, by integrating the three  facts above, we deduce that
		\begin{align}\label{pf-of-thm1.1-30}
			2\tilde{\gamma}&\int_{0}^{1}\int_{\mathbb{R}^{n}}t(1-t)\nabla_{\tilde{A}}{\tilde{u}}\cdot D^{2}\varphi\overline{\nabla_{\tilde{A}}{\tilde{u}}}e^{\gamma |x|^{p}}\ \d x \ \d t
			+4{\tilde{\gamma}}^{3}\int_{0}^{1}\int_{\mathbb{R}^{n}}t(1-t)D^{2}\varphi\nabla \varphi\cdot \nabla\varphi|{\tilde{u}}|^{2}e^{\gamma |x|^{p}}\ \d x \ \d t\notag\\
			&\leqslant c^{*}k^{c_{p,m}}e^{2a_{1}k^{\frac{m}{2-p}}}, \,\,\,\, \mbox{if} \,\,\, k\geqslant k_{0},
		\end{align}
		where $k, k_0,m, c^*,c_{p,m}$ are given by \eqref{MBMV}-\eqref{bounded1-4}. 
		According to the choice of $\varphi$ in \eqref{def-varphi}, we conclude that for all $x\in \mathbb{R}^{n}$,
		\begin{equation}\label{pf-of-thm1.1-31}
			\nabla_{\tilde{A}}{\tilde{u}}\cdot D^{2}\varphi\overline{\nabla_{\tilde{A}}{\tilde{u}}}\geqslant c_{p}(1+|x|)^{p-2}|\nabla_{\tilde{A}}{\tilde{u}}|^{2}.
		\end{equation}
		Thus, combining \eqref{pf-of-thm1.1-30} with \eqref{pf-of-thm1.1-31}, we obtain
		\begin{align}\label{pf-of-thm1.1-32}
			\gamma \int_{0}^{1}\int_{\mathbb{R}^{n}}t(1-t)\frac{1}{(1+|x|)^{2-p}}|\nabla_{\tilde{A}}\tilde{u}(x,t)|^{2}e^{\gamma |x|^{p}}\ \d x \ \d t
			\leqslant c^{*}k^{c_{p,m}}e^{2a_{1}k^{\frac{m}{2-p}}},\,\,\,\, \mbox{if} \,\,\, k\geqslant k_{0},
		\end{align}
		where $k, k_0,m, c^*,c_{p,m}$ as in \eqref{MBMV}-\eqref{bounded1-4}. Then, by \eqref{pf-of-thm1.1-24} and \eqref{pf-of-thm1.1-32}, we prove the desired inequality \eqref{pf-of-thm1.1-33}.
		
		\textbf{Step 2: The lower bounds.} 
		
		We prove that for sufficiently large $R$,  there is an absolute constant $C>0$ such that
		\begin{align}\label{pf-of-thm1.1-37}
			\int_{|x|<\frac{R}{2}}\int_{\frac{3}{8}}^{\frac{5}{8}}|\tilde{u}(x,t)|^{2}\ \d t \ \d x\geqslant \frac{Ce^{-2M}\|u(0)\|_{L^{2}}^{2}}{10},
		\end{align}
		where $M$ is given in \eqref{pf-of-thm1.1-13}.
		Indeed, by \eqref{Appell1-2}, one has
		\begin{align}\label{pf-of-thm1.1-34}
			&\int_{|x|<\frac{R}{2}}\int_{\frac{3}{8}}^{\frac{5}{8}}|\tilde{u}(x,t)|^{2}\ \d t \ \d x\notag\\
			&\ \ \ =\int_{|x|<\frac{R}{2}}\int_{\frac{3}{8}}^{\frac{5}{8}}\bigg|\Big(\frac{\sqrt{\alpha\beta}}{\alpha(1-t)+\beta t}\Big)^{\frac{n}{2}}
			u\Big(\frac{\sqrt{\alpha \beta}x}{\alpha(1-t)+\beta t}, \frac{\beta t}{\alpha(1-t)+\beta t}\Big)\bigg|^{2}\ \d t \ \d x.
		\end{align}
		
		For the time variable $t$, let 
		$$s:=s_k(t)=\frac{\beta t}{\alpha(1-t)+\beta t},$$
		where $\alpha=a_0^{\frac{1}{p}},\beta=k^{\frac{m}{p}}$ (see \eqref{pf-of-thm1.1-12}). Then in the interval $t\in [\frac{3}{8},\frac{5}{8}]$, one has
		\begin{equation}\label{eq-t-s}
			\d t=\frac{(\alpha(1-t)+\beta t)^{2}}{\alpha\beta}\d s=\frac{\beta}{\alpha}\frac{t^2}{s^2}\d s\sim \frac{\beta}{\alpha}\frac{1}{s^2}\d s,
		\end{equation}
		and
		\begin{equation*}
			s_k\bigg(\frac{5}{8}\bigg)-s_k\bigg(\frac{3}{8}\bigg)
			=\frac{\alpha\beta\Big(\frac{5}{8}-\frac{3}{8}\Big)}{\Big(\frac{5}{8}\alpha+\frac{3}{8}\beta\Big)\Big(\frac{3}{8}\alpha+\frac{5}{8}\beta\Big)}
			\sim \frac{\alpha}{\beta},\quad\,\,k\geqslant c_{n},
		\end{equation*}
		where  $c_{n}>0$ is some  large constant. This implies that $s_k\Big(\frac{5}{8}\Big)>s_k\Big(\frac{3}{8}\Big)$. Furthermore, it follows from the definition of $s_k$ that $s_k\Big(\frac{3}{8}\Big)\rightarrow 1$ as $k\rightarrow \infty$ and $s_k\Big(\frac{3}{8}\Big)\geqslant \frac{1}{2}$ for $k\geqslant c_{n}$.
		
		For the spatial variable $x$, let $y=\frac{\sqrt{\alpha\beta}x}{\alpha(1-t)+\beta t}$,  where $\alpha, \beta$ are the same as in  \eqref{pf-of-thm1.1-12}. Then, for $t\in [\frac{3}{8},\frac{5}{8}]$, one has
		\begin{equation}\label{eq-y-x}
			y\sim \sqrt{\frac{\alpha}{\beta}}|x|=\bigg(\frac{a_0}{k^m}\bigg)^{\frac{1}{2p}}|x|.
		\end{equation}
		Taking
		\begin{equation}\label{pf-of-thm1.1-35}
			R\geqslant 2\mu\bigg(\frac{k^m}{a_0}\bigg)^{\frac{1}{2p}},
		\end{equation}
		where $\mu=\mu_k$ is a constant to be determined later (see \eqref{muk}). Thus there exists an absolute constant $C>0$ such that
		\begin{align}\label{pf-of-thm1.1-36}
			\eqref{pf-of-thm1.1-34}&\geqslant \frac{C\beta}{\alpha}\int_{|y|\leqslant \frac{R}{2}(\frac{a_0}{k^m})^{\frac{1}{2p}}}\int_{s_k(\frac{3}{8})}^{s_k(\frac{5}{8})}|u(y,s)|^{2}\ \frac{\d s\ \d y}{s^{2}} \notag\\
			&\geqslant \frac{C\beta}{\alpha}\int_{|y|\leqslant \mu}\int_{s_k(\frac{3}{8})}^{s_k(\frac{5}{8})}|u(y,s)|^{2}\ \d s\ \d y,
		\end{align}
		where in the first inequality, we used \eqref{eq-t-s}  and \eqref{eq-y-x}; in the second inequality, we used \eqref{pf-of-thm1.1-35}. 
		
		To proceed, we observe the following facts:
		
		\textit{Fact 1.}  $\big[s_k\big(\frac{3}{8}\big),s_k\big(\frac{5}{8}\big)\big]\subset \big[\frac{1}{2},1\big]$. Moreover, since 
		$\lim\limits_{k\rightarrow \infty}s_k\big(\frac{3}{8}\big)=1$, thus given $\varepsilon>0$, there exists $k_{0}>0$ such that for any $k\geqslant k_{0}$, $\big[s_k\big(\frac{3}{8}\big),s_k\big(\frac{5}{8}\big)\big]\subset \big[1-\varepsilon,1\big]$.
		
		\textit{Fact 2.} $\big|s_k\big(\frac{5}{8}\big)-s_k\big(\frac{3}{8}\big)\big|\sim \frac{\alpha}{\beta}$ for $k$ sufficiently large. 
		
		\textit{Fact 3.} By the continuity of $\|u(\cdot,s)\|_{L^{2}}$ at $s=1$, there exist $l_{0}=l_{0}(u)$ such that for any $k\geqslant l_{0}$ and for any $s\in \big[s_k\big(\frac{3}{8}\big),s_k\big(\frac{5}{8}\big)\big]$, 
		\begin{equation*}
			\int_{|y|\leqslant \mu}|u(y,s)|^{2}\ \d y\geqslant\frac{C\|u(s)\|_{L^{2}}^{2}}{10} \geqslant\frac{Ce^{-2M}\|u(0)\|_{L^{2}}^{2}}{10}
		\end{equation*}
		holds provided  $\mu\gg 1$, where in the second inequality above, we used the \eqref{pf-of-thm1.1-14}.
		
		Combining \eqref{pf-of-thm1.1-36} with the three facts mentioned above, we obtain \eqref{pf-of-thm1.1-37}.

		\textbf{Step 3: Conclusion of the proof.}

		We  complete the proof by invoking the Carleman inequality as stated in Lemma \ref{carleman estimate}. To achieve this, 
		we choose $\varphi\in C^{\infty}([0,1])$ and $\theta_{R},\, \theta\in C_{0}^{\infty}(\mathbb{R}^{n})$ to satisfy that $0\leqslant \varphi \leqslant 3$,
		\begin{align*}
			\varphi(t)=
			\begin{cases}
				3,& t\in \big[\frac{3}{8},\frac{5}{8}\big],\\
				0,& t\in \big[0,\frac{1}{4}\big]\cup\big[\frac{3}{4},1\big],
			\end{cases}
		\end{align*}
		\begin{align*}
			\theta_{R}(x)=
			\begin{cases}
				1,& |x|\leqslant R-1,\\
				0,& |x|>R,
			\end{cases}
		\end{align*}
		and
		\begin{align*}
			\theta(x)=
			\begin{cases}
				1,& |x|\geqslant 2,\\
				0,& |x|\leqslant 1.
			\end{cases}
		\end{align*}
		Set
		\begin{equation}\label{pf-of-thm1.1-43}
			g(x,t)=\theta_{R}(x)\theta\Big(\frac{x}{R}+\varphi(t)e_{1}\Big)\tilde{u}(x,t).
		\end{equation}
		Based on the definitions of $\varphi(t), \theta_{R}(x), \theta(x)$, we have the following facts:
		
		\noindent(i) If $|x|\leqslant \frac{R}{2}$, $t\in \big[\frac{3}{8},\frac{5}{8}\big]$, then 
		$\big|\frac{x}{R}+\varphi(t)e_{1}\big|\geqslant \frac{5}{2}>2$.
		Consequently, we have
		\begin{equation}\label{eq-g(x,t)}
			g(x,t)=\tilde{u}(x,t),\,\,\mbox{and}\,\,e^{\sigma |\frac{x}{R}+\varphi(t)e_{1}|^{2}}\geqslant e^{\frac{25}{4}\sigma}.
		\end{equation}
		
		\noindent (ii) If $|x|\geqslant R$ or $t\in \big[0,\frac{1}{4}\big]\cup\big[\frac{3}{4},1\big]$, then $g(x,t)=0$, and 
		$$
		\mbox{supp} \ g\subseteq \Big\{(x,t)\in \R^n\times [0,1]: |x|\leqslant R, \ \frac{1}{4}\leqslant t \leqslant \frac{3}{4}\Big\}\cap \Big\{\big|\frac{x}{R}+\varphi(t)e_{1}\big|\geqslant 1\Big\}.
		$$ 
		Let us define $\xi=\frac{x}{R}+\varphi(t)e_{1}$. With this definition, we write
		\begin{align}\label{pf-of-thm1.1-44}
			(\i\partial_{t}+\Delta_{\tilde{A}}+\tilde{V})g
			&=[\theta(\xi)(2\nabla \theta_{R}(x)\cdot \nabla_{\tilde{A}}\tilde{u}+\tilde{u}\Delta \theta_{R}(x))+2\nabla \theta(\xi)\cdot \nabla \theta_{R}(x)\tilde{u}]\notag\\
			&\ \ \ +\theta_{R}(x)[2R^{-1}\nabla\theta(\xi)\cdot\nabla_{\tilde{A}}\tilde{u}+R^{-2}\tilde{u}\Delta\theta(\xi)+\i\varphi'\partial_{x_1}\theta(\xi)\tilde{u}]\notag\\
			&=:F_1+F_2.
		\end{align}
		In view of the definitions of $\varphi(t), \theta_{R}(x), \theta(x)$, we  derive the following support conditions 
		\begin{equation*}
			\mbox{supp} \ F_1\subset \Big\{(x,t)\in \mathbb{R}^{n}\times [0,1]:R-1\leqslant |x|\leqslant R,\,\, \frac{1}{32}\leqslant t\leqslant\frac{31}{32}\Big\},
		\end{equation*}
		and
		\begin{equation*}
			\mbox{supp} \ F_2\subset \Big\{(x,t)\in \mathbb{R}^{n}\times [0,1]:1\leqslant\Big|\frac{x}{R}+\varphi(t)e_{1}\Big|\leqslant 2\Big\}.
		\end{equation*}
		
		From now on, we assume that in \eqref{pf-of-thm1.1-35},
		\begin{equation}\label{muk}
			\mu=k^{\frac{m(p-1)}{p(2-p)}},
		\end{equation}
		and take
		\begin{equation}\label{eq-R-s}
			R=2a_0^{-\frac{1}{2p}}k^{\frac{m}{2(2-p)}}\gg k^{\frac{m}{2p}}.
		\end{equation}
		We apply the Carleman inequality \eqref{carleman1-2} and take
		\begin{align}\label{eq-sigma}
			\sigma=cR^{2},
		\end{align}
		where $c$ is the constant appeared in \eqref{carleman1-2}. This, together with the identity \eqref{pf-of-thm1.1-44}, yields that
		\begin{align}\label{pf-of-thm1.1-45}
			R\|e^{\sigma|\frac{x}{R}+\varphi(t)e_{1}|^{2}}g\|_{L^{2}(\d x \ \d t)}
			&\leqslant c_{n} \|e^{\sigma|\frac{x}{R}+\varphi(t)e_{1}|^{2}}(\i\partial_t+\Delta_{\tilde{A}})g\|_{L^{2}(\d x \ \d t)}\notag\\
			&\leqslant c_{n} \|e^{\sigma|\frac{x}{R}+\varphi(t)e_{1}|^{2}}\tilde{V}g\|_{L^{2}(\d x \ \d t)}
			+c_{n} \|e^{\sigma|\frac{x}{R}+\varphi(t)e_{1}|^{2}}F_{1}\|_{L^{2}(\d x \ \d t)}\notag\\
			&\ \ \ +c_{n} \|e^{\sigma|\frac{x}{R}+\varphi(t)e_{1}|^{2}}F_{2}\|_{L^{2}(\d x \ \d t)}\notag\\
			&=:G_{1}+G_{2}+G_{3}.
		\end{align}
		Now we will address each term $G_j$ for $1\leqslant j\leqslant 3$ separately.
		
		For the term $G_1$, when $(x,t)\in \mathbb{R}^{n}\times [\frac{1}{32},\frac{31}{32}]$,  it follows from  \eqref{pf-of-thm1.1-5} that
		$$
		|\tilde{V}(x,t)|\leqslant 32^2 \frac{\alpha}{\beta}\|V\|_{L^{\infty}}= 32^2 \Big(\frac{a_0}{k^m}\Big)^{\frac{1}{p}}\|V\|_{L^{\infty}}.
		$$
		By combining the above inequality with the fact \eqref{eq-R-s}, we deduce that
		\begin{equation*}\label{pf-of-thm1.1-41}
			R\gg \|\tilde{V}\|_{L^{\infty}\big(\mathbb{R}^{n}\times [\frac{1}{32},\frac{31}{32}]\big)}.
		\end{equation*}
		This allows us to absorb the term $G_{1}$ into the left hand side of \eqref{pf-of-thm1.1-45}. 
		
		For the term $G_2$, we introduce the  following quantity:
		\begin{equation}\label{pf-of-thm1.1-42}
			\omega(R)=\Big(\int_{\frac{1}{32}}^{\frac{31}{32}}\int_{R-1\leqslant|x|\leqslant R}(|\tilde{u}(x,t)|^{2}+|\nabla_{\tilde{A}}\tilde{u}(x,t)|^{2})\ \d x\ \d t\Big)^{\frac{1}{2}},
		\end{equation}
		where $R$ satisfies \eqref{eq-R-s}.
		Given that $|\frac{x}{R}+\varphi(t)e_{1}|\leqslant 4$, we obtain
		\begin{equation}\label{eq-G2}
			G_{2}\leqslant c_{n}\omega(R)e^{16\sigma}.
		\end{equation}
		
		For the term $G_{3}$, since $1\leqslant|\frac{x}{R}+\varphi(t)e_{1}|\leqslant 2$, so 
		\begin{equation}\label{eq-G3}
			G_{3}\leqslant c_{n}e^{4\sigma}\|\ |\tilde{u}|+|\nabla_{\tilde{A}}\tilde{u}|\ \|_{L^{2}\big(\{|x|\leqslant R\}\times [\frac{1}{32},\frac{31}{32}]\big)}.
		\end{equation}
		Notice that 
		\begin{equation*}
			\|\tilde{u}(t)\|_{L^{2}}^{2}=\|u(s)\|_{L^{2}}^{2}, \ \ \ s=\frac{\beta t}{\alpha(1-t)+\beta t},
		\end{equation*}
		then by \eqref{pf-of-thm1.1-14}, we get
		\begin{equation*}\label{pf-of-thm1.1-38}
			\int_{|x|<R}\int_{\frac{1}{32}}^{\frac{31}{32}}|\tilde{u}(x,t)|^{2}\ \d t\ \d x\leqslant \|u(0)\|_{L^{2}}^{2}e^{2M}.
		\end{equation*}
		Meanwhile, by \eqref{pf-of-thm1.1-33}, we have
		\begin{align*}\label{pf-of-thm1.1-39}
			\int_{\frac{1}{32}}^{\frac{31}{32}}\int_{|x|\leqslant R}|\nabla_{\tilde{A}}\tilde{u}(x,t)|^{2}\ \d x\ \d t
			&\leqslant c_{n}\int_{\frac{1}{32}}^{\frac{31}{32}}\int_{|x|\leqslant R}t(1-t)\frac{(1+|x|)^{2-p}}{(1+|x|)^{2-p}}e^{\gamma |x|^{p}}|\nabla_{\tilde{A}}\tilde{u}(x,t)|^{2}\ \d x\ \d t\notag\\
			&\leqslant c_{n}\gamma^{-1}R^{2-p}c^{*}k^{c_{p,m}}e^{2a_{1}k^{\frac{m}{2-p}}}\leqslant c^{*}k^{c_{p,m}}e^{2a_{1}k^{\frac{m}{2-p}}},
		\end{align*}
		where we used \eqref{pf-of-thm1.1-12} and \eqref{eq-R-s} in the last inequality.
		Thus, for $k\geqslant k_{0}$ sufficiently large, we have
		\begin{equation*}\label{pf-of-thm1.1-40}
			\int_{|x|<R}\int_{\frac{1}{32}}^{\frac{31}{32}}(|\tilde{u}(x,t)|^{2}+|\nabla_{\tilde{A}}\tilde{u}(x,t)|^{2})\ \d t\ \d x\leqslant c^{*}k^{c_{p,m}}e^{2a_{1}k^{\frac{m}{2-p}}},
		\end{equation*}
		where $k, k_0,m, c^*,c_{p,m}$ are given by \eqref{MBMV}-\eqref{bounded1-4}. 
		This, together with \eqref{eq-G3}, yields
		\begin{equation}\label{eq-G3-3}
			G_{3}\leqslant c^{*}e^{4\sigma}k^{c_{p,m}}e^{2a_{1}k^{\frac{m}{2-p}}}.
		\end{equation}
		
		Regarding the left hand side of \eqref{pf-of-thm1.1-45}, observe that by the fact \eqref{eq-g(x,t)} and the lower bound \eqref{pf-of-thm1.1-37} established in \textbf{Step 2}, we derive that
		\begin{equation*}
			\sqrt{\frac{C}{10}}e^{-M}e^{\frac{25}{4}\sigma}\|u(0)\|_{L^{2}}\leqslant R\|e^{\sigma|\frac{x}{R}+\varphi(t)e_{1}|^{2}}g\|_{L^{2}(\d x \ \d t)}.
		\end{equation*}
		By inserting this as well as the upper bounds \eqref{eq-G2} and \eqref{eq-G3-3} concerning $G_2, G_3$ into \eqref{pf-of-thm1.1-45}, we obtain that 
		\begin{align*}
			\sqrt{\frac{C}{10}}e^{-M}e^{\frac{25}{4}\sigma}\|u(0)\|_{L^{2}}
			\leqslant c_{n}\omega(R)e^{16\sigma}+c^{*}k^{c_{p,m}}e^{4\sigma}e^{2a_{1}k^{\frac{m}{2-p}}}.
		\end{align*}
		Take $a_{1}=2ca_0^{-\frac{1}{p}}$, thus we have $\sigma=2a_{1}k^{\frac{m}{2-p}}$ by \eqref{eq-sigma}. Hence, the above inequality leads to the following lower bound of $\omega(R)$:
		\begin{align}\label{pf-of-thm1.1-47}
			\omega(R)\geqslant c_{n}e^{-M}\|u(0)\|_{L^{2}}e^{-10\sigma}=c_{n}e^{-M}\|u(0)\|_{L^{2}}e^{-20a_{1}k^{\frac{m}{2-p}}},
		\end{align}
		for $k\geqslant k_{0}$ sufficiently large.
		
		We now proceed to determine an upper bound for $\omega(R)$ based on the upper bound established in \textbf{Step 1}. Specifically, one has
		\begin{align}\label{pf-of-thm1.1-48}
			\omega^{2}(R)&=\int_{\frac{1}{32}}^{\frac{31}{32}}\int_{R-1\leqslant|x|\leqslant R}(|\tilde{u}(x,t)|^{2}+|\nabla_{\tilde{A}}\tilde{u}(x,t)|^{2})\ \d x\ \d t\notag\\
			&\leqslant c_{n}e^{-\gamma (R-1)^{p}}\sup_{[0,1]}\|e^{\frac{\gamma}{2} |x|^{p}}\tilde{u}(t)\|_{L^{2}}^{2}\notag\\
			&\ \ \ +c_{n}\gamma^{-1}R^{2-p}e^{-\gamma (R-1)^{p}}\int_{\frac{1}{32}}^{\frac{31}{32}}\int_{R-1\leqslant|x|\leqslant R}\frac{t(1-t)}{(1+|x|)^{2-p}}e^{\gamma |x|^{p}}|\nabla_{\tilde{A}}\tilde{u}(x,t)|^{2}\ \d x\ \d t\notag\\
			&\leqslant c^{*}k^{c_{p,m}}e^{2a_{1}k^{\frac{m}{2-p}}}e^{-\gamma (R-1)^{p}},
		\end{align}
		where in the first inequality, we used \eqref{pf-of-thm1.1-33} and the fact $R-1 \leqslant |x|\leqslant R$; while in the second inequality,  we used \eqref{pf-of-thm1.1-12} and \eqref{eq-R-s}.
		
		Thus  we obtain
		\begin{align}\label{pf-of-thm1.1-49}
			c_{n}\|u(0)\|_{L^{2}}^{2}e^{-2M}
			&\leqslant c^{*}k^{c_{p,m}}e^{42a_{1}k^{\frac{m}{2-p}}-\gamma (R-1)^{p}}\notag\\
			&\leqslant c^{*}k^{c_{p,m}}e^{42a_{1}k^{\frac{m}{2-p}}-2^{\frac{p}{2}}a_{0}^{\frac{1}{2}}
				\left(\frac{a_{1}}{c}\right)^{\frac{p}{2}}k^{\frac{m}{2-p}}+O\big(k^{\frac{m}{2(2-p)}}\big)},
		\end{align}
		where in the first inequality, we used \eqref{pf-of-thm1.1-47}, \eqref{pf-of-thm1.1-48}; in the second inequality, we used $\gamma=(k^m a_0)^{\frac{1}{2}}$ (see \eqref{pf-of-thm1.1-12}),  the identity
		$$\sigma=cR^{2}=2a_{1}k^{\frac{m}{2-p}},\quad \frac{1}{2}+\frac{p}{2(2-p)}=\frac{1}{2-p},$$ 
		as well as the fact that $(R-1)^{p}=R^p+O(R^{p-1})$. Therefore, if 
		$$42a_{1}<2^{\frac{p}{2}}a_{0}^{\frac{1}{2}}\left(\frac{a_{1}}{c}\right)^{\frac{p}{2}},$$ 
		i.e., $a_{0}a_{1}^{p-2}>(42)^{2}\left(\frac{c}{2}\right)^{p}$,
		then  it follows that $u\equiv 0$ by letting $k\rightarrow \infty$ in \eqref{pf-of-thm1.1-49}. Therefore the proof of  Theorem \ref{morgan for p-p} is complete.
		$\hfill\Box$
		
		\subsection{Proof of Theorem \ref{morgan for p-q}}\label{sec-4-2}\,
		
		We  proceed to verify that the conditions of Theorem \ref{morgan for p-p} are fulfilled.
		
		First, under the assumption \eqref{cor1.1-1}, it is evident that condition \eqref{thm1.1-6} is satisfied with $a_0=\frac{\alpha^p}{p}$.
		
		Second, by assumption \eqref{cor1.1-1} again, we have 
		\begin{equation*}
			\int_{\R^n}|u(x,1)|^{2}e^{2b|x|^{q}}\ \d x<\infty,\,\,\,\,b=\frac{\beta^{q}}{q}.
		\end{equation*}
		Observe that
		\begin{equation*}
			\int_{\R^n}|u(x,1)|^{2}e^{2k^m|x|^{p}}\ \d x\leqslant \|f_{pq}(\cdot)\|_{L^{\infty}}\int_{\R^n}|u(x,1)|^{2}e^{2b|x|^{q}}\ \d x.
		\end{equation*}
		where $f_{pq}(x):=e^{2k^m|x|^{p}-2b|x|^{q}}$, $x\in \mathbb{R}^{n}$. A simple computation shows that
		\begin{equation*}
			\|f\|_{L^{\infty}}=f\Big|_{|x|=\big(\frac{k^m p}{bq}\big)^{\frac{1}{q-p}}}=e^{2a_{1}k^{\frac{mq}{q-p}}},
		\end{equation*}
		where
		\begin{equation}\label{pf-of-cor1.2-1}
			a_{1}=\frac{c_{p}}{b^{\frac{p}{q-p}}},\,\,\, c_p=\Big[\Big(\frac{p}{q}\Big)^{\frac{p}{q-p}}-\Big(\frac{p}{q}\Big)^{\frac{q}{q-p}}\Big].
		\end{equation}
		Thus, the condition  \eqref{thm1.1-7} is satisfied with $a_1$ given by \eqref{pf-of-cor1.2-1}.
		
		Third, since we have assumed 
		\begin{equation*}
			\alpha\beta>N_{p},
		\end{equation*}
		this, together with the explicit expressions of $a_0, a_1, b$, as well as the fact $\frac{q(2-p)}{q-p}=1$ given above, yields that 
		\begin{equation}\label{pf-of-cor1.2-2}
			a_0a_1^{p-2}>M_p,\,\,\,M_p=p^{-1}q^{-\frac{p}{q}}c_p^{p-2}N_p^{p}.
		\end{equation}
		Thus the condition \eqref{thm1.1-8} is also fulfilled.
		
		Therefore, we can apply Theorem \ref{morgan for p-p} and the proof of Theorem \ref{morgan for p-q} is complete.
	\end{proof}

	\subsection{Proof of Corollary \ref{apply morgan}}\label{sec-4-3}\,
	
	Let
	\begin{equation*}
		u(x,t)=u_{1}(x,t)-u_{2}(x,t),
	\end{equation*}
	and
	\begin{equation*}
		V(x,t)=\frac{F(u_1, \bar{u}_1)-F(u_2, \bar{u}_2)}{u_1-u_2},
	\end{equation*}
	Then Theorem \ref{morgan for p-q} yields the result.
	$\hfill\Box$

	\section*{Acknowledgements}
	S. Huang was supported by the National Natural Science Foundation of China under grants 12171178
	and 12171442.

	
\end{document}